\documentclass{amsproc}

\newtheorem{theorem}{Theorem}[section]
\newtheorem{lemma}[theorem]{Lemma}
\newtheorem{col}{Corollary}[section]
\theoremstyle{definition}
\newtheorem{definition}[theorem]{Definition}

\theoremstyle{remark}
\newtheorem{remark}[theorem]{Remark}

\numberwithin{equation}{section}

\begin{document}

\title[Bernstein spaces, sampling, Boas interpolation formulas ]{Notes  on Bernstein spaces, sampling,   Boas interpolation formulas and their extensions to Banach spaces}

\author{Isaac Z. Pesenson}
\address{Department of Mathematics, Temple University,
 Philadelphia,
PA 19122}

\email{pesenson@temple.edu}


\begin{abstract}

This paper is essentially a survey on several classical  results of harmonic analysis and their recent extensions to Banach spaces. The first part of the paper  is a summary of some important results  in such topics as  Bernstein spaces, Shannon-type sampling, Riesz and Boas interpolation formulas.  The second  part  contains   extensions of these ideas to   Banach spaces equipped with one-parameter uniformly bounded group of operators of class $C_{0}$. 
 
\end{abstract}

\maketitle

\tableofcontents

\section{Introduction}

This paper contains a  survey on several classical  results related to analysis in Bernstein spaces along with extensions of them to abstract settings. The first part of the paper  is a summary of some basic results  in such topics as  Bernstein spaces, Shannon-type sampling, Riesz and Boas interpolation formulas. The corresponding  sections   \ref{Bern-spaces}-\ref{Higher order 3} are significantly influenced by the books \cite{Akh, B2, BN, Nik}, and especially by the papers written by P. Butzer and his collaborators \cite{BSS-1}-\cite{BSS-3}.
The second  part, which consists of sections  \ref{Bernstein-vectors}-\ref{H-Boas formulas},    contains   extensions of these classical   ideas to   Banach spaces equipped with one-parameter uniformly bounded groups of operators. The corresponding results  were published in   \cite{KPes}, \cite{Pes88}--\cite{Pes23}. 

The paper is organized as follows. In section \ref{Bern-spaces} we introduce distributional Fourier transform $\mathcal{F}$, its inverse $\mathcal{F}^{-1}$, and define Bernstein spaces $B_{\sigma}^{p}, \>\>\sigma>0,$  as such which comprised of all $f \in L_{p}(\mathbb{R})\cap C^{\infty}(\mathbb{R}),\>\>1\leq p\leq \infty, $ whose distributional Fourier transform has support (in the sense of distributions) in  $[-\sigma, \>\sigma]$. 
 In this section we also introduce the Poisson summation formula and the Riesz Interpolation formula for trigonometric polynomials. We also provide a proof of the    Boas interpolation formula for functions in $B_{\sigma}^{p}, \>\>1\leq p\leq \infty, \>\sigma>0$.  As a consequence of this formula one obtains the Bernstein inequality for functions in Bernstein spaces.
 In section \ref{more-prop}, exploring the Bernstein inequality, we show that for a function in a Bernstein space its $L_{p}$-norm controls a discrete norm with equally spaces nodes
\begin{equation}\label{PP-0}
\left(h\sum_{k}|f(x_{k})|^{p}\right)^{1/p}\leq (1+h\sigma)\|f\|_{p},
\end{equation}
 for $f\in B_{\sigma}^{p},\>1\leq p\leq \infty,\>\sigma>0,\>x_{k}=hk, \>k\in \mathbb{Z}, h>0.$ 
  In turn, this fact implies the following  important inequality    
\begin{equation}\label{Nik-0}
\|f\|_{p}\leq \sup_{x\in \mathbb{R} }
\left(h\sum_{k} |f(x- x_{k})|^{p}\right)^{1/p}\leq (1+h\sigma)\|f\|_{p},
\end{equation}
where the same notations as in (\ref{PP-0}) are used.
  This inequality is used to establish the following continuous embeddings
 $$
B_{\sigma}^{1}\subset B_{\sigma}^{p}\subset B_{\sigma}^{q}\subset B_{\sigma}^{\infty},\>\>1\leq p\leq q\leq \infty.
$$
 Section \ref{sampling-theor} contains several Shannon-type sampling theorems. The proofs of all of them essentially rely on the inequality (\ref{PP-0}). The higher order interpolation formulas which generalize the Boas formula are collected in sections \ref{Higher order 2} and  \ref{Higher order 3}. More material relevant to the topics discussed in sections  \ref{Bern-spaces}-\ref{Higher order 3}  can be found, for example, in  \cite{And}-\cite{ Bard2}, \cite{Ganz}, \cite{Gro},  \cite{Mon1, Mon2}, \cite{Schm}, \cite{Stein}-\cite{ AZ}.

 In  section \ref{Bernstein-vectors} we introduce  Bernstein vectors in Banach spaces which are equipped with one-parameter uniformly bounded group of operators of  class $C_{0}$. Namely, if $T(t),\> t\in \mathbb{R}$, is  such a group whose infinitesimal operator is $D$, then  the generalized Bernstein space ${\bf B}_{\sigma}(D),\>\sigma>0$, is a set of all vectors $f$ for which the Bernstein inequalities hold true:
 \begin{equation}
\|D^{k}f\|\leq \sigma^{k}\|f\|,\>\>k\in \mathbb{N}.
\end{equation}
 It is very important that we can show that the set $\bigcup_{\sigma>0}{\bf B}_{\sigma}(D)$ is dense in the Banach space. Moreover, we illustrate that this property does not hold for a general one-parameter semigroup of operators. In this section we also give other equivalent descriptions of the abstract Bernstein spaces and discuss the Bernstein vectors in Hilbert spaces. We also prove a Landau-Kolmogorov-Stein-type inequality associated with the group of operators $T(t)$.
 
 Section \ref{sampling in Banach spaces} contains sampling theorems for trajectories $T(t)f,\>t\in \mathbb{R}$, where $f\in {\bf B}_{\sigma}(D)$. These theorems provide some explicit formulas for reconstruction of entire  trajectories $T(t)f, \>f\in {\bf B}_{\sigma}(D), \>t\in \mathbb{R},$ from an infinite set of samples $T(t_{k})f,\>k\in \mathbb{Z},$ taken from the set 
 $\{T(t)f\}_{t\in \mathbb{R}}$.

Note, that if $T(t)$ is a group of operators than every trajectory $T(t)f$ is completely determined by any single sample $T(t_{0})f,$ because for any $t\in \mathbb{R},$
$$
T(t)f=T(t-t_{0})\left(T(t_{0})f\right).
$$

Our results  have,  however, a different  nature. They represent a trajectory $T(t)f$ as a "linear combination" of a countable number of its samples. Results of this kind can be useful when  the members of the  group  $T(t)$ are  unknown and only samples $T(t_{k})f, \>k\in \mathbb{Z},$ of a trajectory  $T(t)f$ are given.

It seems to be very interesting, that no matter how complicated one-parameter group can be (think, for example, about a Schr\"{o}dinger operator $D=-\Delta+V(x)$ and the corresponding group $e^{itD}$ in $L_{2}(\mathbb{R}^{d})$) the formulas in  section       \ref{sampling in Banach spaces}         are universal in the sense that the  coefficients and the  sets of sampling points are independent on the group.

Section \ref{RB in Banach} is about higher order Boas-type formulas for one-parameter groups of operators.
 Sections \ref{Analysis with HT}, \ref {H-sampling}, \ref{H-Boas formulas}, apply the developed abstract results to the case of the Discrete Hilbert Transform in $l^{2}$.

\section{Bernstein spaces}\label{Bern-spaces}

 \subsection{Bernstein spaces}

For a Schwartz function $\varphi\in S(\mathbb{R})$ its Fourier transform is defined as
$$
\mathcal{F}\varphi(\xi)=\frac{1}{\sqrt{2\pi}}\int_{\mathbb{R}}\varphi(x)e^{-i x\xi}dx,
$$ 
and the inverse Fourier transform is defined as
$$
\mathcal{F}^{-1}\varphi(x)=\frac{1}{\sqrt{2\pi}}\int_{\mathbb{R}}\varphi(\xi)e^{i x\xi}d\xi.
$$
These definitions are extended to distributions $f\in S^{'}(\mathbb{R})$ by the formulas
$$
\langle\mathcal{F}f, \varphi\rangle =\langle f, \mathcal{F}\varphi \rangle,\>\>\> \langle\mathcal{F}^{-1}f, \varphi\rangle =\langle f, \mathcal{F}^{-1}\varphi \rangle,\>\>\>\varphi\in S(\mathbb{R}),
$$
where for $f\in S^{'}(\mathbb{R}),\>\>\varphi\in S(\mathbb{R})$, 
$$
\langle f, \varphi\rangle =\int_{\mathbb{R}} f(x)\varphi(x)dx.
$$
The Paley-Wiener Theorem says that a function $f\in L^{p}(\mathbb{R}),\>\>1\leq p\leq \infty, $ belongs to
$B^p_\sigma$ if and only if its distributional Fourier transform $\mathcal{F}f$
has  support $[-\sigma,\sigma]$ in the sense of distributions.
The space $B^2_\sigma$ is called the Paley-Wiener space and is
denoted by $PW_\sigma .$  
 \subsection{Poisson summation formula}

Let $L_{\lambda}^{p}(\mathbb{R})$ be a space of $\lambda$- periodic functions on $\mathbb{R}$ with the norm
$$
\|f\|_{L_{\lambda}^{p}}=\left(\frac{1}{\lambda}\int_{\lambda/2}^{-\lambda/2}|f(t)|^{p}dt\right)^{1/p}.
$$
With every  $f\in L_{\lambda}^{1}$, one can associate the Fourier series 
$$
 \sum_{k\in \mathbb{Z}}c_{\lambda}^{k}(f)e^{i2k\pi/\lambda}
$$
where
$$
c_{\lambda}^{k}(f)=\frac{1}{\lambda}\int_{-\lambda/2}^{\lambda/2}f(t) e^{-i2k\pi/\lambda}dt.
$$
The generalized Parseval formula for $f_{1}, f_{2}\in L_{\lambda}^{2}(\mathbb{R})$  has the following form
\begin{equation}
\frac{1}{\lambda}\int f_{1}(t)\overline{f_{2}(t)}dt=\sum_{k}c_{\lambda}^{k}(f_{1})\overline{c_{\lambda}^{k}}(f_{2}).
\end{equation}
For a sufficiently fast decaying function   $f$ its periodization $ f^{*}=(\lambda/\sqrt{2\pi})\sum f(\cdot +\lambda k)$  belongs to $L_{\lambda}^{1}(\mathbb{R})$ and the corresponding Fourier series is given by
$$
\sum_{k}\left(\mathcal{F}f\right) \left(2k\pi/\lambda\right)e^{i2k\pi t/\lambda}.
$$
If $f\in L^{1}(\mathbb{R})$ is absolutely continuous and $f^{'}\in L^{1}(\mathbb{R})$  then one 
 has actual equality
 \begin{equation}\label{Poisson}
 (\lambda/\sqrt{2\pi})\sum_{k\in \mathbb{Z}} f(t +\lambda k)=\sum_{k\in \mathbb{Z}}(\mathcal{F}f) \left(2k\pi/\lambda\right)e^{i2k\pi t/\lambda},
 \end{equation}
 which is known as the Poisson summation formula.

\subsection{Classical Riesz and Boas interpolation formulas}

 A trigonometric polynomial of order $N$  is a function of the form 
\begin{equation}\label{trig}
 P_{N}(x)=\sum_{k=-N}^{N}c_{k}e^{ikx},\>\>\>\>x\in \mathbb{R},
 \end{equation}
 which is $2\pi$-periodic and it 
can be treated as a function in $B_{N}^{\infty}$. In this case, due to periodicity there exists a finite interpolation formula
 \begin{equation}\label{R1}
P_{N}^{'}(x)=\frac{1}{4N}\sum_{k=1}^{2N}(-1)^{k+1}\frac{1}{\sin^{2}\frac{x_{k}}{2}}P_{N}(x+x_{k}),\>\>\>x_{k}=\frac{2k-1}{2N}\pi,
 \end{equation}
 which was discovered by Riesz in \cite{R1, R2}. 
For $P_{N}(x)=\sin N x$  one obtains 
 $$
 N=\frac{1}{4N}\sum_{k=1}^{2N}\frac{1}{\sin^{2}\frac{x_{k}}{2}},
 $$
 and (\ref{R1}) implies the Bernstein inequality for
  trigonometric polynomials
  $$
  \sup_{x\in (-\pi, \pi)} |P_{N}^{'}(x)|\leq N  \sup_{x\in (-\pi, \pi)}|P_{N}(x)|.
  $$
In \cite{B1} Boas published what is known today as the Boas interpolation formula for functions in $B_{\sigma}^{p},\>\>1\leq p\leq\infty,$
\begin{equation}\label{Boas11}
\frac{(-1)^{k-1}}{(k-1/2)^2} f\left(
x+\frac{\pi}{\sigma}(k-1/2)\right),
\end{equation}
where convergence of the series holds in the norm of $L^{p}(\mathbb{R})$. The idea of the original proof of this formula 
is very transparent and we replicate it now.

On the interval $(-\sigma, \sigma)$ we consider the function $i\xi e^{i\pi \xi/2\sigma}$ and its Fourier series
$$
i\xi e^{i\pi \xi/2\sigma}=\sum_{k}c_{k}^{\sigma}e^{\left(i\pi k\xi/\sigma\right)},
$$
where
$$
c_{k}^{\sigma}=\frac{i}{2\sigma}\int_{-\sigma}^{\sigma}\tau e^{i \left(\frac{1}{2}-k\right)\pi \tau/\sigma}d\tau=
$$
$$
-\frac{1}{\sigma}\int_{0}^{\sigma}\tau \sin\left(\left(k-1/2\right)\frac{\pi \tau}{\sigma} \right)d\tau=\frac{\sigma(-1)^{k-1}}{\pi^{2}\left(k-1/2\right)^{2}}.
$$
From here one obtains
\begin{equation}\label{it}
i\xi=\frac{4\sigma}{\pi^{2}}\sum_{k\in \mathbb{Z}}
\frac{(-1)^{k}}{(2k+1)^{2}}e^{i\pi (2k+1)\xi/2\sigma},  \>\>\>\xi\in (-\sigma, \sigma),
\end{equation}
where the series converges uniformly on  $(-\sigma, \sigma)$. 
If $f\in B_{\sigma}^{p}$ then 
$\mathcal{F}f $ is supported on $(-\sigma, \sigma)$ and substitution of (\ref{it}) into the formula
$$
f^{'}(x)=\mathcal{F}^{-1}\left(i\xi \mathcal{F}f(\xi)\right),
$$
gives (\ref{Boas11}).

\subsection{Bernstein inequality}

When $f(x)=\sin x\in B_{1}^{\infty}$ and $x=0,$ the formula (\ref{Boas11}) implies  the following identity
\begin{equation}\label{identity}
\sum_{k\in \mathbb{Z}}\frac{1}{(k-1/2)^2}=\pi^2.
\end{equation}
The formula (\ref{Boas11}) shows that if $f\in B_{\sigma}^{p}$ then its derivative is also in $B_{\sigma}^{p}$. Moreover, 
formulas (\ref{Boas11}) and  (\ref{identity}) along with the fact that 
 the  norm of $L_{p}(\mathbb{R})$ is invariant with respect to translations 
 provide proof of the  famous Bernstein inequality for functions in $B_{\sigma}^{p},\>1\leq p\leq \infty,\>\sigma>0,$ 
\begin{equation}\label{Bern-1}
\left\|f^{'}\right\|_p\leq \sigma \left\|f\right\|_p, 1\leq p\leq \infty, \> \sigma>0 .
\end{equation}
Uniform convergence  of the series in (\ref{Boas11}) allows differentiate both sides of it to get 
$$
f^{''}(x)=\frac{\sigma}{\pi^2}\sum_{k\in \mathbb{Z}}
\frac{(-1)^{k-1}}{(k-1/2)^2} f^{'}\left(
x+\frac{\pi}{\sigma}(k-1/2)\right),
$$
which leads to 
$$
\left\|f^{''}\right\|_p\leq \sigma^{2} \left\|f\right\|_p, \>\>\>1\leq p\leq \infty, \> \sigma>0.
$$
Continue this way we are coming to the Bernstein inequality for any $m\in \mathbb{N}$
\begin{equation}
\left\|f^{(m)}\right\|_p\leq \sigma \left\|f\right\|_p, \>\>\>1\leq p\leq \infty, \> \sigma>0.
\end{equation}

\section{More properties of functions in Bernstein spaces}\label{more-prop}

The following statement is a part of the famous Plancherel-Polya frame inequalities.
\begin{lemma}\label{PP}
$f\in B_{\sigma}^{p},\>1\leq p\leq \infty,\>h>0,\>x_{k}=kh,\>k\in \mathbb{Z},$ then
\begin{equation}\label{Nik}
\sup_{x\in \mathbb{R} }
\left(h\sum_{k} |f(x_{k})|^{p}\right)^{1/p}\leq (1+h\sigma)\|f\|_{p}.
\end{equation}

\end{lemma}

\begin{proof}
The case $p=\infty $ is trivial.
For  $1\leq p<\infty$, one has
$$
\|f\|_{p}^{p}=\int_{\mathbb{R}}|f(x)|^{p}dx=\sum_{k}\int_{x_{k}}^{x_{k+1}}|f(x)|^{p}dx=h\sum_{k}|f(y_{k})|^{p},
$$
where $y_{k}\in (x_{k}, x_{k+1})$. 
Now, by using the inequality
$$
\left|\>\|a\|-\|b\|\>\right|\leq \|a-b\|,
$$
we have
$$
\left|\left(h\sum_{k}|f(y_{k})|^{p}\right)^{1/p}-\left(h\sum_{k}|f(x_{k})|^{p}\right)^{1/p}\right|\leq
\left(    h \sum_{k}f(y_{k} ) -h\sum_{k} f(x_{k})                      \right)^{1/p}\leq
$$
$$
\left(h\sum_{k}|f(y_{k})-f(x_{k})|^{p}\right)^{1/p} \leq \left(h\sum_{k}\left |\int_{x_{k}}^{x_{k+1}} f^{'}(\tau)d\tau\right|^{p}\right)^{1/p}\leq
$$
$$
\left(h\sum_{k}\int_{x_{k}}^{x_{+1}}|f^{'}|^{p}d\tau h^{p/q}\right)^{1/p}=      h\|f^{'}\|_{p}\leq h\sigma\|f\|_{p},                                                                                        
$$
where we utilized the Bernstein inequality (\ref{Bern-1}). From here for $x_{k}=hk, \>k\in \mathbb{Z}, h>0,$ one obtains
$$
\left(h\sum_{k}|f(x_{k})|^{p}\right)^{1/p}=
$$
$$
\left(\left(h\sum_{k}|f(x_{k})|^{p}\right)^{1/p}-\left(h\sum_{k}|f(y_{k})|^{p}\right)^{1/p}\right)+
$$
$$
\left(h\sum_{k}|f(y_{k})|^{p}\right)^{1/p}\le (1+h\sigma)\|f\|_{p}, \>\>\>f\in B_{\sigma}^{p},\>\>1\leq p\leq\infty.
$$
Thus we proved for $f\in B_{\sigma}^{p},\>x_{k}=hk, \>k\in \mathbb{Z}, h>0,$ the next inequality 
\begin{equation}\label{PP}
\left(h\sum_{k}|f(x_{k})|^{p}\right)^{1/p}\leq (1+h\sigma)\|f\|_{p}.
\end{equation}
\end{proof}
Lemma is proven.

\begin{theorem}\label{N}
If $f\in B_{\sigma}^{p},\>1\leq p\leq \infty,\>h>0,\>x_{k}=kh,\>k\in \mathbb{Z},$ then
\begin{equation}\label{Nik}
\|f\|_{p}\leq \sup_{x\in \mathbb{R} }
\left(h\sum_{k} |f(x- x_{k})|^{p}\right)^{1/p}\leq (1+h\sigma)\|f\|_{p}
\end{equation}
\end{theorem}
\begin{proof}
The case $p=\infty $ is trivial.
For a fixed $t\in \mathbb{R}$ apply the inequality (\ref{PP}) to the function 
$$
x\mapsto f(t-x),
$$
to obtain the next inequality for every $t\in \mathbb{R}$
$$
\left(h\sum_{k}|f(t-x_{k})|^{p}\right)^{1/p}\leq (1+h\sigma)\|f\|_{p},
$$
which is proving the right had side of the inequality (\ref{Nik}).

This inequality shows that for an  $f\in B_{\sigma}^{p},\>1\leq p\leq\infty,$ the middle part  of (\ref{Nik})  is finite and 
$$
\int_{\mathbb{R}}|f|^{p}dx=\sum_{k}\int_{x_{k}}^{x_{k+1}}|f|^{p}dx=\sum_{k}\int_{0}^{h}|f(x_{k}-x)|^{p}dx=
$$
$$
\int_{0}^{h}\sum_{k}|f(x_{k}-x)|^{p}dx\leq h\sup_{x\in \mathbb{R}}\sum_{k}|f(x_{k}-x)|^{p}<\infty.
$$
Theorem is proven.
\end{proof}

\begin{col} For $1\leq p\leq q\leq \infty,$ the following continuous embeddings hold
$$
B_{\sigma}^{1}\subset B_{\sigma}^{p}\subset B_{\sigma}^{q}\subset B_{\sigma}^{\infty}.
$$
\end{col}
\begin{proof}
Applying the inequality
$$
\left(\sum_{k}a_{k}^{q}\right)^{1/q}\leq \left(\sum_{k}a_{k}^{p}\right)^{1/p},
\>\>\>1\leq p\leq q\leq \infty,
$$
along with (\ref{Nik}) one obtains for $f\in B_{\sigma}^{p},\>1\leq p\leq \infty,\>h>0,\>x_{k}=kh,\>k\in \mathbb{Z},$
$$
\|f\|_{q}\leq \left(h\sum_{k} |f(x- x_{k})|^{q}\right)^{1/q}\leq 
$$
$$
\left(\sum_{k} \left(h^{1/q}|f(x- x_{k})|\right)^{q}\right)^{1/q}\leq
\left(\sum_{k} \left(h^{1/q}|f(x- x_{k})|\right)^{p}\right)^{1/p}\leq
$$
$$
h^{1/q-1/p}\left(h\sum_{k} |f(x-x_{k})|^{p}\right)^{1/p}\leq  h^{1/q-1/p}(1+h\sigma)\|f\|_{p}.
$$
Corollary is proven.
\end{proof}

\section{Sampling theorems}\label{sampling-theor}

Let us remind that  the function $\text{ sinc}$ is defined for $x\in \mathbb{R}$ as follows
$$
\text{ sinc}\> x=\begin{cases}
                \frac{\sin x }{x},  &\text {x $\neq 0$}\\
                1                     ,& \text {x$=0$}.
                \end{cases}
                $$

\begin{theorem}\label{sampling1}
For any $f\in B_{\sigma}^{p},\>\>1\leq p<\infty,\>\>\sigma>0,$ the following formula holds
\begin{equation}\label{S1}
f(x)=\sum_{k\in \mathbb{Z}}  f\left(\frac{k\pi}{\sigma}\right)\textit{sinc}\left(\frac{\sigma}{\pi}x-k\right),
\end{equation}
and also  for each $m\in \mathbb{Z}$ one has
\begin{equation}\label{derivative}
f^{(m)}(x)=\sum_{k\in \mathbb{Z}}f\left(\frac{k\pi}{\sigma}\right)\left(\frac{d}{dx}\right)^{m} sinc\left(\frac{\sigma x}{\pi}-k\right),
\end{equation}
where in the both cases the series converge absolutely and uniformly on the entire $\mathbb{R}$.
\end{theorem}
\begin{proof}
Because the function $\rm sinc$ belongs to $L_{p}(\mathbb{R})$ for any $p>1$, using the right hand side of the inequality in Theorem \ref{N} and the H\"older inequality we obtain absolute and uniform convergence of the series in (\ref{S1}).

Since $\mathcal{F}(f)$ has support in $(-2\pi\sigma, \>2\pi\sigma)$, 
the Poisson summation formula (\ref{Poisson}) applied to $f\in B_{2\pi \sigma}^{1}$ with $\lambda=1/\sigma, \>t=0,$ gives
\begin{equation}\label{ApPoisson}
\frac{1}{\sqrt{2\pi}\sigma}\sum_{k}f\left(k/\sigma\right)=\mathcal{F}(f)(0).
\end{equation}
For  $f_{1}\in B_{\pi\sigma}^{p}, \>f_{2}\in B_{\pi\sigma}^{q},\>1/p+1/q=1, $ consider the convolution 
$$
f_{1}\ast f_{2}(t)=\frac{1}{\sqrt{2\pi}}\int_{\mathbb{R}}f_{1}(\tau f_{2}(t-\tau)d\tau,
$$
and apply  the formula ( \ref{ApPoisson}) to the product 
$$
\tau\mapsto f_{1}(\tau)f_{2}(t-\tau),
$$
to obtain
 \begin{equation}\label{conv}sinc
f_{1}\ast f_{2}(t)=\frac{1}{\sqrt{2\pi}\sigma}\sum_{k}f_{1}(k/\sigma)f_{2}(t-k/\sigma),
 \end{equation}
where the next relation was used 
$$
\mathcal{F}(f)(0)=\frac{1}{\sqrt{2\pi}}\int_{\mathbb{R}}f(\tau )d\tau.
$$
since the continuous convolution is commutative
$$
f_{1}\ast f_{2}(t)=f_{2}\ast f_{1}(t),
$$
the formula (\ref{conv}) shows the following property 
$$
\frac{1}{\sqrt{2\pi}\sigma}\sum_{k}f_{1}(k/\sigma)f_{2}(t-k/\sigma)=\frac{1}{\sqrt{2\pi}\sigma}\sum_{k}f_{1}(t-k/\sigma)f_{2}(k/\sigma).
$$
Now, we apply the last formula to the functions$f_{1}(x)=f(x)$ and $f_{2}(x)=sinc\left(\frac{\sigma}{\pi} x\right)$ to have

$$
\sum_{k}f\left(\frac{k\pi}{\sigma}\right)
sinc
\left( \frac{\sigma}{\pi}
\left(t-\frac{k\pi}{\sigma}\right)\right)=\sum_{k} f\left(t-\frac{k\pi}{\sigma}\right)sinc(k)=f(x).
$$
It proves the formula (\ref{S1}). 

To prove (\ref{derivative}) we note that 
by applying the Leibnitz formula  
one can obtain for $m\in \mathbb{N},$ 
\begin{equation}\label{sinc-1}
sinc^{(m)}(x)=\left(\frac{1}{\pi x}\sin \>\pi x \right)^{(m)}=\sum_{j=0}^{m} 
\left( \begin{array}{c} m \\ j \end{array} \right)   
 (\sin \pi x)^{(j)}\left(\frac{1}{\pi x}\right)^{(m-j)}=
$$
$$
\sum_{j=0}^{m}\left( \begin{array}{c} m \\ j \end{array} \right)
\pi^{j}\sin\left(\pi x+\frac{j\pi}{2}\right) \frac{(-1)^{m-j}(m-j)!}{\pi x^{m-j+1}}=  
$$
$$
\frac{(-1)^{m}m!}{\pi x^{m+1}}\sum_{j=0}^{m}\sin\left(\pi x+\frac{j\pi}{2}\right)\frac{(-1)^{j}(\pi x)^{j}}{j!},
\end{equation}
which shows that for every $m\in \mathbb{Z}$ the function $sinc^{(m)}(x)$ belongs to $L_{p}(\mathbb{R})$ for $p>1$.
According to the  H\"older's inequality  and (\ref{Nik}) the series in (\ref{derivative}) is absolutely and uniformly convergent which allows term-by-term differentiation of (\ref{S1}).
Theorem is proven.
\end{proof}

\begin{remark}
The Theorem \ref{sampling1} is not necessary true for functions in $B_{\sigma}^{\infty}$.
 Indeed, consider $\sin z,\>\>z\in \mathbb{C},$ for which 
 \begin{equation}\label{sin}
 |\sin z|=\frac{1}{2}|e^{i z}-e^{-i z}|\leq\frac{1}{2}(e^{-Im z}+e^{ Im z})\leq e^{|Im z|},
 \end{equation}
 thus $\sin x \in B_{1}^{\infty}$.
 For this function the formula (\ref{sin}) takes the form
 $$
 \sin x=\sum_{k\in \mathbb{Z}}  \sin (k\pi) sinc\left(\frac{1}{\pi}x-k\right),
 $$
 where the right-hand side is identical zero.
\end{remark}

The case $B_{\sigma}^{\infty} $ is covered in the following sampling theorem \cite{BSS-1}.

\begin{theorem}\label{smaller sampling}
If $f\in B_{\sigma}^{p},\>\>1\leq p<\infty,$ or $f\in B_{\tau }^{\infty}$ for some $0\leq \tau<\sigma$, then for $z\in \mathbb{C}$ one has
\begin{equation}\label{S2}
f(z)=\sum_{k\in \mathbb{Z}}  f\left(\frac{k\pi}{\sigma}\right)sinc\left(\frac{\sigma}{\pi}z-k\right),
\end{equation}
where the series uniformly converges on compact subsets of $\mathbb{C}$.
\end{theorem}

The next  theorems can be found in \cite{BFHSS}.

\begin{theorem}(Valiron-Tschakalov sampling and interpolation formula)\label{classicalVT}

For $f\in B_{\sigma}^{\infty},\>\>\sigma>0,$ and for $z\in \mathbb{C}$ one has
\begin{equation}\label{VT-1}
f(z)=
$$
$$
zf^{'}(0)sinc\left(\frac{\sigma z}{\pi}\right)+f(0)sinc\left(\frac{\sigma z}{\pi}\right)+\sum_{k\in \mathbb{Z} \setminus{ 0}} f\left(\frac{k\pi}{\sigma}\right)\frac{\sigma z}{k\pi} sinc\left(\frac{\sigma z}{\pi}-k\right),
\end{equation}
where convergence is absolute and uniform on compact subsets of $\mathbb{C}$.
\end{theorem}
\begin{proof}
If $f\in B_{\pi}^{\infty}$ then the function  
$$
g(z)=\begin{cases}
                \frac{f(z)-f(0)}{z},  &\text {z $\neq 0$}\\
               f^{'}(0)                     ,& \text {z$=0$}.
                \end{cases}
                $$
belongs to $B_{\pi}^{2}$ and we can use Theorem \ref{S1} to obtain
for $z\neq 0$
$$
\frac{f(z)-f(0)}{z}=f^{'}(0)sinc \>z+\sum_{k\in \mathbb{Z}\setminus \{0\}}\frac{f(k)-f(0)}{k} sinc(z-k).
$$
For $f=sinc$ it gives

$$
\frac{sinc\>z-1}{z}=-\sum_{k\in \mathbb{Z}\setminus \{0\}}\frac{sinc(z-k)}{k}.
$$
These two formulas imply (\ref{VT-1}).

\end{proof}

The next two statements can be found in \cite{BSS-1}.
\begin{theorem}(Classical sampling formula). For $f\in B_{\sigma}^{2},\>\>\sigma>0$, we have
$$
f(z)=\sum_{k\in \mathbb{Z}} f\left(\frac{k\pi}{\sigma}\right) \text{sinc}\left(\frac{\sigma z}{\pi}-k\right),\>\>z\in \mathbb{C},
$$
where the convergence being absolute and uniform in strips of bounded width parallel to the real line, in particular on the entire real line and on compact subsets in $\mathbb{C}$.
\end{theorem}
The next one is called the Whittaker-Kotel{'}nikov-Shannon sampling theorem.
\begin{theorem}\label{WKS}
If $f\in B_{\sigma}^{2}, \>\sigma>0,$ then for each $m\in \mathbb{Z}$ 
$$
f^{(m)}(x)=\sum_{k\in \mathbb{Z}}f\left(\frac{k\pi}{\sigma}\right)\left(\frac{d}{dx}\right)^{m} \text{sinc}\left(\frac{\sigma x}{\pi}-k\right),
$$
where the series converges absolutely and uniformly on the entire $\mathbb{R}$ and allso in the $L_{2}(\mathbb{R})$-norm.
\end{theorem}

\section{Higher order interpolation formulas with coefficients of order $O(k^{-2})$.}\label{Higher order 2}
Starting with (\ref{sinc-1}) and 
expressing $\sin (\pi x+j\pi/2)$ in terms of $\sin \pi x$ and $\cos \pi x$ gives the next formula
\begin{equation}\label{formula}
sinc^{(m)}x=
$$
$$
\frac{(-1)^{m}m!}{\pi x^{m+1}}\left( \sin \pi x\sum_{\nu=0}^{[m/2]}\frac{(-1)^{\nu}(\pi x)^{2\nu}}{(2\nu)!}-   \cos \pi x \sum_{\nu=0}^{[(m-1)/2]}\frac{(-1)^{\nu}(\pi x)^{2\nu+1}}{(2\nu+1)!} \right).
\end{equation}
As it follows from the power series expansion 
$$
sinc \>x=\sum_{j\in \mathbb{N}\cup \{0\}} \frac{(-1)^{j}}{(2j+1)!}(\pi x)^{2j},
$$
the right hand side of (\ref{formula}) can be continuously extended to $x=0$ by setting $sinc^{(m)}(0)=0$ for  odd $m$, and $sinc ^{(m)}(0)=(-1)^{m/2}\pi^{m}/(m+1)$ for even $m$.

To obtain the next two identities (\ref{A})  and (\ref{B}) one has to notice that they follow from (\ref{formula}) by taking in account that 
$$
\sin \left(\pi(k-1/2)\right)=\cos \pi k=(-1)^{k}
$$ 
and 
$$
\cos \left(\pi(k-1/2)\right)=\sin \pi k=0.
$$
Thus we have
\begin{equation}\label{A}
(-1)^{k+1} \rm sinc ^{(2m-1)}\left(\frac{1}{2}-k\right)=
\frac{(2m-1)!}{\pi(k-\frac{1}{2})^{2m}}\sum_{j=0}^{m-1}\frac{(-1)^{j}}{(2j)!}\left(\pi\left(k-\frac{1}{2}\right)\right)^{2j},
\end{equation}
and for $m\in \mathbb{N},\>\>\>k\in \mathbb{Z}\setminus \{0\},$
\begin{equation}\label{B}
(-1)^{k+1} \rm sinc ^{(2m)}(-k)=\frac{(2m)!}{\pi k^{2m+1}}\sum_{j=0}^{m-1}\frac{(-1)^{j}(\pi k)^{2j+1}}{(2j+1)!}.
\end{equation}
We introduce notations
$$
A_{m,k}=(-1)^{k+1} \rm sinc ^{(2m-1)}\left(\frac{1}{2}-k\right),\>\>\>
k\in \mathbb{Z},\>\>\>m\in \mathbb{N},
$$
 and 
$$
B_{m,k}=(-1)^{k+1} \rm sinc ^{(2m)}(-k),\>\>\>m\in \mathbb{N},\>\>\>k\in \mathbb{Z}\setminus {0}
$$
and
$$
B_{m,0}=(-1)^{m+1} \frac{\pi^{2m}}{2m+1},\>\>\>m\in \mathbb{N}.
$$
We notice that the following relations hold
\begin{equation}\label{relations}
\left(\frac{\sigma}{\pi}\right)^{2m-1}\sum_{k\in \mathbb{Z}}\left|A_{m,k}\right|=\sigma^{2m-1},\>\>\>\>\>
 \left(\frac{\sigma}{\pi}\right)^{2m}\sum_{k\in \mathbb{Z}}\left|B_{m,k}\right|=\sigma^{2m}.
 \end{equation}
 We will also need the relation 
\begin{equation}\label{0}
\sum_{k\in \mathbb{Z}} sinc^{(m)}(x-k)=0,\>\>m\in \mathbb{N}.
\end{equation}
 To justify it we note that the Fourier transform of $ sinc\>x$ is the characteristic function of the interval $(-\pi, \pi)$. 
 By the Fourier inversion formula one has
 $$
sinc^{(m)}(x)=\frac{1}{2\pi}\int_{-\pi}^{\pi}(i\xi)^{m}e^{i\xi x}d\xi,
$$
which implies 
$$
\sum_{k\in  \mathbb{Z}}sinc^{(m)}(x-k)e^{i k \xi}=\sum_{k\in \mathbb{Z}}\left(\frac{1}{2\pi}\int_{-\pi}^{\pi}(i\xi)^{m}e^{-i(k-x)\xi}d\xi\right).
$$
 It is not difficult to check that the right hand side here is $(i\xi)^{m}e^{i x \xi}$. Thus we obtain the formula
 $$
 \sum_{k\in  \mathbb{Z}}sinc^{(m)}(x-k)e^{i k \xi}=(i\xi)^{m}e^{i x \xi},
 $$
 which for $\xi=0$ gives the relation (\ref{0}).

 In the interesting papers \cite{BSS-2} and \cite{BSS-3} among other important results the Boas formula (\ref{Boas11}) was extended  to higher order of derivatives.

\begin{theorem}\label{H-R-B-1}

For $f\in B_{\sigma}^{\infty},\>\>\>\sigma>0,$ the following formulas hold 
\begin{equation}\label{Butzer1}
f^{(2m-1)}(x)=\left(\frac{\sigma}{\pi}\right)^{2m-1}\sum_{k\in \mathbb{Z}}(-1)^{k+1}A_{m,k}f\left(x+\frac{\pi}{\sigma}(k-\frac{1}{2})\right), \>\>\>m\in \mathbb{N},
\end{equation}
\begin{equation}\label{Butzer2}
f^{(2m)}(x)=\left(\frac{\sigma}{\pi}\right)^{2m}\sum_{k\in \mathbb{Z}}(-1)^{k+1}B_{m,k}f\left(x+\frac{\pi}{\sigma}k\right), \>\>\>m\in \mathbb{N}.
\end{equation}
If $f\in B_{\sigma}^{p},\>\>1\leq p\leq \infty,$ the the series on the right converge in the norm of $L_{p}(\mathbb{R})$.

\end{theorem}

\begin{proof}
Consider $f\in B_{\pi}^{\infty}$ and define $f_{1}\in B_{\pi}^{2}$ as 
$$
f_{1}(x)=\frac{f(x)-f(0)}{x},\>\>x\in \mathbb{R}\setminus \{0\},
$$
and $f_{1}(0)=f^{'}(0)$.
One has for $x\in \mathbb{R}, \>m\in \mathbb{N},$
$$
f^{(m)}(x)=\left( x f_{1}(x)\right)^{(m)}=mf^{m-1}_{1}(x)+x f_{1}^{m}(x).
$$
Because  $f_{1}\in B_{\pi}^{2}$ we can apply Theorem \ref{sampling1} to the terms on the right side to get 

\begin{equation}\label{fla deriv}
f^{(m)}(x)=m\sum_{k\in \mathbb{Z}}f_{1}(k) sinc^{(m-1)}(x-k)+x\sum_{k}f_{1}(k)sinc^{(m)}(x-k),\>m\in \mathbb{N}.
\end{equation}

Next, for $x=1/2$,  changing $m$ to $2m-1$ and applying 
$$
(2m-1)sinc^{(2m-2)}\left(\frac{1}{2}-k\right)=\left(k-\frac{1}{2}\right)sinc^{(2m-1)}\left(\frac{1}{2}-k\right),\>k\in \mathbb{Z},
$$
which follows from the formula (\ref{formula}), 
we obtain
$$
f^{(2m-1)}\left(\frac{1}{2}\right)=(2m-1)\sum_{k\in \mathbb{Z}} f_{1}(k)sinc^{(2m-2)}\left(\frac{1}{2}-k\right)+\frac{1}{2}\sum f_{1}(k)sinc^{(2m-1)}\left(\frac{1}{2}-k\right)=
$$
$$
\sum f_{1}(k)\left(k-\frac{1}{2}\right) sinc^{(2m-1)}\left(\frac{1}{2}-k\right)+\frac{1}{2}\sum f_{1}(k)sinc^{(2m-1)}\left(\frac{1}{2}-k\right)=
$$
$$
\sum_{k\in \mathbb{Z}}\left(f(k)-f(0)\right) sinc ^{(2m-1)}\left(\frac{1}{2}-k\right).
$$
By using the formula (\ref{0})
we finally have for $m\in \mathbb{N}$:
\begin{equation}\label{odd}
f^{(2m-1)}\left(\frac{1}{2}\right)=\sum_{k\in \mathbb{Z}}f(k)sinc ^{(2m-1)}\left(\frac{1}{2}-k\right)=\sum_{k\in \mathbb{Z}}f(k)(-1)^{k+1}A_{m,k}.
\end{equation}
In order to finish the proof for odd derivatives, we pick a function $f\in B_{\sigma}^{\infty},\>\sigma>0,$
and apply (\ref{odd}) to the function
$$
\tau\mapsto f\left(\frac{\pi}{\sigma}\tau+x-\frac{\pi}{2\sigma}\right).
$$
In the case of even derivatives we start with (\ref{fla deriv}), replace $m$ by $2m$ and set $x=0$. 
One has
\begin{equation}\label{sinc}
2m \>sinc ^{(2m-1)}(-k)=k \>sinc^{(2m)}(-k),
\end{equation}
and together with (\ref{0}) it gives
$$
f^{(2m)}(0)=2m\sum f_{1}(k) sinc^{(2m-1)}(-k)=\sum k f_{1}(x) sinc^{(2m)}(-k)=
$$
$$
\sum\left(f(k)-f(0)\right) sinc^{(2m)}(-k)=\sum_{k}f(k)(-1)^{k+1}B_{m,k}, \>\>m\in \mathbb{N}.
$$
To complete the proof one has to apply the last equation to the function 
$$
\tau \mapsto  f\left(\frac{\pi}{\sigma}\tau+x\right).
$$
Theorem is proven.

\end{proof}

Theorem \ref{H-R-B-1} and relations (\ref{relations}) along with fact that 
 the  norm of $L_{p}(\mathbb{R})$ is invariant with respect to translations 
 provide a proof of the  Bernstein inequality for functions in $B_{\sigma}^{p},\>1\leq p\leq \infty,$ 
$$
\left\|f^{(n)}\right\|_p\leq \sigma^n \left\|f\right\|_p, \quad \mbox{ for all } n\in \mathbb{N}.
$$

\section{Higher order interpolation formulas with coefficients of order $O(k^{-3})$.}\label{Higher order 3}

\begin{theorem} For every $f\in B_{\sigma}^{\infty},\>\>\>\sigma>0 $ and $m\in \mathbb{N}$ the following formulas hold
\begin{equation}\label{Butzer10}
f^{(2m)}(t)=(-1)^{m}\sigma^{2m}+ \frac{2m\sigma^{2m}}{\pi^{2m}}\sum_{k}(-1)^{k+1}\frac{A_{m,k}}{k-1/2}f\left(t+\frac{\pi}{\sigma} \left(k-1/2\right)\right).
\end{equation}

\begin{equation}\label{Butzer20}
f^{(2m+1)}(t)=
$$
$$
-\frac{(2m+1)\sigma^{2m}}{\pi^{2m}}
B_{m,0}f^{'}(t)+\frac{(2m+1)\sigma^{2m+1}}{\pi^{2m+1}}\sum_{k\neq 0}(-1)^{k+1}\frac{B_{m,k}}{k}f(t+\pi k/\sigma).
\end{equation}

\end{theorem}

\begin{proof}
For a function $f\in B_{\sigma}^{\infty}$ the function 

$$
f_{1}(x)=\begin{cases}
                \frac{f(x)-f(0)}{x},  &\text {x $\neq 0$}\\
               f^{'}(0)                     ,& \text {x$=0$}.
                \end{cases}
                $$
belongs to $B_{\sigma}^{\infty}$. By comparison of Taylor series of $f$ and $f^{'}$ at zero we have
$$
\frac{f_{1}^{(r)}(0)}{r!}=\frac{f^{(r+1)}(0)}{(r+1)!},\>\>\>r\in \mathbb{N}.
$$
By using (\ref{Butzer2}) to $f_{1}$ at zero, we obtain
$$
f^{(2m)}(0)=\frac{2m\sigma^{2m}}{\pi^{2m}}\sum_{k}(-1)^{k+1}A_{m,k}\frac{f\left(\frac{\pi}{\sigma}(k-1/2)\right)-f(0)}{k-1/2}.
$$
For $f(x)=\cos \pi x$ this formula holds with $\sigma=\pi$ and gives
$$
(-1)^{m}\pi^{2m}=-2m\sum_{k}(-1)^{k+1}\frac{A_{m,k}}{k-1/2}.
$$
After all 
$$
f^{(2m)}(0)=(-1)^{m}\sigma^{2m}+ \frac{2m\sigma^{2m}}{\pi^{2m}}\sum_{k}(-1)^{k+1}\frac{A_{m,k}}{k-1/2}f\left(\frac{\pi}{\sigma} \left(k-1/2\right)\right).
$$
Applying this formula to the function $x\mapsto f(x+t)$ and setting $x$ to zero gives (\ref{Butzer10}).
Similarly, using (\ref{Butzer1}) for  $f_{1}$ at  $x=0, $ we obtain
$$
f^{(2m+1)}(0)=
$$
$$
-\frac{(2m+1)\sigma^{2m}}{\pi^{2m}}
B_{m,0}f^{'}(0)+\frac{(2m+1)\sigma^{2m+1}}{\pi^{2m+1}}\sum_{k\neq 0}(-1)^{k+1}\frac{B_{m,k}}{k}\frac{f(\pi k/\sigma)-f(0)}{k}.
$$
For the $sinc$ function and $\sigma=\pi$ it gives
$$
(2m+1)\sum_{k\neq 0}(-1)^{k+1}\frac{B_{m,k}}{k}=0,
$$
and then
$$
f^{(2m+1)}(0)=
$$
$$
-\frac{(2m+1)\sigma^{2m}}{\pi^{2m}}
B_{m,0}f^{'}(0)+\frac{(2m+1)\sigma^{2m+1}}{\pi^{2m+1}}\sum_{k\neq 0}(-1)^{k+1}\frac{B_{m,k}}{k}f(\pi k/\sigma).
$$
By applying this formula to the shifted function $x\mapsto f(x+t)$ at $x=0$, we obtain (\ref{Butzer20}).
Theorem is proven.
\end{proof}

\section{Bernstein vectors  in Banach spaces}\label{Bernstein-vectors}

We assume that $D$ is the generator of a one-parameter group of
isometries $e^{ tD}$ in a Banach space $E$ with the norm $\|\cdot
\|$ (for precise definitions see \cite{BB}, \cite{K}). The notations $\mathcal{D}^{k}$ will be used for  the domain of $D^{k}$, and notation $\mathcal{D}^{\infty}$ for $\bigcap_{k\in \mathbb{N}}\mathcal{D}^{k}$.

\begin{definition}
The  subspace of exponential vectors $\mathbf{E}_{\sigma}(D), \>\>\sigma\geq 0,$ is defined as  a set of all vectors $f$ in  $\mathcal{D}^{\infty}$  for which there exists a constant $C(f,\sigma)>0$ such that 
\begin{equation}\label{Exp}
\|D^{k}f\|\leq C(f,\sigma)\sigma^{k}, \>\>k\in \mathbb{N}.
\end{equation}
 
\end{definition}
Note, that every $\mathbf{E}_{\sigma}(D)$ is clearly a linear subspace of $E$. What is really  important is the fact that union of all $\mathbf{E}_{\sigma}(D)$ is dense in $E$ (Corollary \ref{Density}). 

\begin{remark}\label{rem}
If $D$ generates a strongly continuous bounded semigroup (but not a group) then the set  $\bigcup_{\sigma\geq 0}\mathbf{E}_{\sigma}(D)$ may not be  dense in $E$. 

Indeed,  consider a strongly continuous bounded semigroup $T(t)$ in $L_{2}(0,\infty)$ defined for every $f\in L_{2}(0,\infty)$ as $T(t)f(x)=f(x-t),$ if $\>x\geq t$ and $T(t)f(x)=0,$ if $\>\>0\leq x<t$.   Inequality (\ref{Exp}) implies that if $ f\in \mathbf{E}_{\sigma}(D)$ then for any $g\in L_{2}(0,\infty)$ the function $\left<T(t)f,\>g\right>$ is analytic in $t$. Thus if $g$ has compact support then $\left<T(t)f,\>g\right>$ is zero for  all  $t$ which implies that $f$ is zero. In other words in this case every space $\mathbf{E}_{\sigma}(D)$ is trivial. 

\end{remark}

\begin{definition}
The Bernstein subspace
 $\mathbf{B}_{\sigma}(D), \>\>\sigma\geq 0,$ is defined as  a set of all vectors $f$ in $E$ 
 which belong to $\mathcal{D}^{\infty}$  and for which 
\begin{equation}\label{Bernstein}
\|D^{k}f\|\leq \sigma^{k}\|f\|, \>\>k\in \mathbb{N}.
\end{equation}
\end{definition}

At this point it is not even clear that $\mathbf{B}_{\sigma}(D), \>\>\sigma\geq 0,$ is a subspace.
\begin{lemma}[\cite{Pes00}]\label{basic}
  Let $D$ be the generator of a one-parameter group of isometries 
$e^{tD }$ in a Banach space $E$.  If for
some $f\in E$ there exists an $\sigma
>0$ such
that the quantity

$$
\sup_{k\in N} \|D^{k}f\|\sigma ^{-k}=R(f,\sigma)
$$
 is finite, then $R(f,\sigma)\leq
\|f\|.$
\end{lemma}

\begin{proof}
  By assumption $\|D^{r}f\|\leq R(f,\sigma)\sigma ^{r}, r\in \mathbb{N}$.  Now for
any complex number $z$ we have
$$
 \left\|e^{zD}f\right\|=\left\|\sum ^{\infty}_{r=0}(z^{r}D^{r}g)/r!\right\|\leq R(f,\sigma) \sum
^{\infty}_{r=0}|z|^{r}\sigma^{r}/r!=R(f,\sigma)e^{|z|\sigma}.
$$
It implies that for any functional $h^{*}\in E^{*}$ the scalar function
$\left<e^{zD}f,h^{*}\right>$ is an entire function
 of exponential type $\sigma $ which is bounded on the real axis  $\mathbb{R}$
by the constant $\|h^{*}\| \|f\|$.
 An application of the classical Bernstein inequality gives

$$\left\|\left<e^{tD}D^{k}f,h^{*}\right>\right\|_{C(\mathbb{R})}=\left\|\left(\frac{d}{dt}\right)^{k}\left<e^{tD}f,h^{*}\right>\right\|_{
C(\mathbb{R})} \leq\sigma^{k}\|h^{*}\| \|f\|.$$
The last one gives for $t=0$

$$ \left|\left<D^{k}f,h^{*}\right>\right|\leq \sigma ^{k} \|h^{*}\| \|f\|.$$
Choosing $h^{*}$ such that $\|h^{*}\|=1$ and $\left<D^{k}f,h^{*}\right>=\|D^{k}f\|$ we obtain
the inequality $\|D^{k}f\|\leq
 \sigma ^{k} \|f\|, k\in N$, which gives

$$R(f,\sigma)=\sup _{k\in \mathbb{N} }(\sigma ^{-k}\|D^{k}f\|)\leq \|f\|.$$
Lemma is proved.
\end{proof}
This lemma and it's proof have important corollaries.
\begin{col}
For every $\sigma\geq 0$ one has 
$$
\mathbf{B}_{\sigma}(D)= \mathbf{E}_{\sigma}(D). 
$$
\end{col}

\begin{proof}
The inclusion  $
\mathbf{B}_{\sigma}(D)\subset \mathbf{E}_{\sigma}(D) , \>\>\>\sigma\geq 0, 
$ is obvious. The opposite inclusion follows from the previous lemma. 
\end{proof}
The next one follows from the proof of the Lemma.

\begin{col}[\cite{Pes14}]\label{key-col}
  Let $D$ be a generator of  one-parameter group of isometries 
$e^{tD }$ in a Banach space $E$.  Then for every $f\in {\bf B}_{\sigma}(D),\>\sigma\geq 0,$ 
the following holds
\begin{enumerate}

\item For every functional $g\in E^{*}$ the the function of the real variable $t$ defined as $\langle e^{tD}f, g^{*}\rangle$ belongs to the regular Bernstein space $B_{\sigma}^{\infty}(\mathbb{R})$.

\item The abstract valued function $e^{tD}f$ is bounded on the real line and has extension to the complex plain as an entire function of the exponential type $\sigma$.

\end{enumerate}
\end{col}

\begin{col}
Every $\mathbf{B}_{\sigma}(D)$ is a closed linear subspace of $E$ which is invariant under $D$.
\end{col}

\begin{proof}
Linearity of $\mathbf{B}_{\sigma}(D)$ follows from the previous theorem and linearity of $\mathbf{E}_{\sigma}(D)$. If $f_{j}$ converges in $E$ to $f$ then the sequence $\{f_{j}\}$ is fundamental in $E$ and the Bernstein inequality implies that $\{D^{k}f_{j}\},\>\>k\in \mathbb{N},$ is fundamental. since $D^{k},\>\>k\in \mathbb{N},$ is a closed operator it implies that $f$ belongs to the domain of $D^{k},\>\>k\in \mathbb{N},$ and the Bernstein inequality 
$$
\|D^{k}f\|\leq \sigma^{k}\|f\|
$$
holds for any $k\in \mathbb{N}$.
If $f\in \mathbf{B}_{\sigma}(D)$ then the right-hand side of the formula 
$$
Df=\lim_{t\rightarrow 0}\frac{e^{tD}f-f}{t}
$$
belongs to $\mathbf{B}_{\sigma}(D)$ and since $\mathbf{B}_{\sigma}(D)$ is closed the element $Df$ is also in $\mathbf{B}_{\sigma}(D)$.  Hence the corollary is proved.

\end{proof}

 \subsection{Density of the abstract Bernstein spaces}

 Let's introduce the operator
 $\>\>
 \Delta^{m}_{s}f=(I-e^{sD})^{m}f, \>\>m\in \mathbb{N}, 
 $
 and  the modulus of
continuity \cite{BB}
$$
\Omega_{m}(f,s)=\sup_{|\tau|\leq
s}\left\|\Delta^{m}_{\tau}f\right\|\label{dif}.
$$

\begin{theorem}\label{Density}
The set  $\bigcup_{\sigma\geq 0}\mathbf{B}_{\sigma}(D)$ is dense in $E$.
\end{theorem}

\begin{proof}

Note that if  $\phi\in L_{1}(\mathbb{R}),\>\>\>\|\phi\|_{1}=1,$ is an entire function of exponential
type $\sigma$ then for any $f\in E$ the vector
$$
g=\int _{-\infty}^{\infty}\phi(t)e^{tD}fdt
$$
belongs to $\mathbf{B}_{\sigma}(D).$ Indeed,  for every real $\tau$ we
have
$$
e^{\tau
D}g=\int_{-\infty}^{\infty}\phi(t)e^{(t+\tau)D}fdt=\int_{-\infty}^{\infty}\phi(
t-\tau)e^{tD}fdt.
$$
 Using this formula we can extend the abstract function $e^{\tau D}g$ to the
complex plane as
$$
e^{zD}g=\int_{-\infty}^{\infty}\phi(t-z)e^{tD}fdt.
$$
 since by assumption  $h$ is an entire function of exponential
type $\sigma$ and $\|\phi\|_{L_{1}(\mathbb{R})}=1$
 we have
$$
\|e^{zD}g\|\leq
\|f\|\int_{-\infty}^{\infty}|\phi(t-z)|dt\leq\|f\|e^{\sigma|z|}.
 $$
 This inequality  implies that $g$ belongs to $\mathbf{B}_{\sigma}(D)$.

Let
 $$
h(t)=a\left(\frac{\sin (t/4)}{t}\right)^{4}
$$
 and
$$
a=\left(\int_{-\infty}^{\infty}\left(\frac{\sin
(t/4)}{t}\right)^{4}dt\right)^{-1}.
$$
Function $h$ will have the
 following properties:

\begin{enumerate}

\item $h$ is an even nonnegative entire function of exponential
type one;

\item $h$ belongs to $L_{1}(\mathbb{R})$ and its
$L_{1}(\mathbb{R})$-norm is $1$;

\item  the integral
\begin{equation}\int_{-\infty}^{\infty}h(t)|t|dt
\end{equation} 
is finite.
\end{enumerate}
Consider
 the following vector 
\begin{equation}
\mathcal{R}_{h}^{\sigma}(f)=
\int_{-\infty}^{\infty}h(t)e^{\frac{t}{\sigma}D
} fdt=\int_{-\infty}^{\infty}h(t\sigma)e^{tD}fdt,\label{id}
 \end{equation}
since the function $h(t)$ has exponential type one the function
$h(t\sigma)$ has the type $\sigma$. It  implies (by the previous) that $\mathcal{R}_{h}^{\sigma}(f)$ belongs to $\mathbf{B}_{\sigma}(D)$.

The modulus of
continuity is defined as in \cite{BB}
$$
\Omega(f,s)=\sup_{|\tau|\leq
s}\left\|\Delta_{\tau}f\right\|,\>\>\>\>
 \Delta_{\tau}f=(I-e^{\tau D})f. \label{dif}
$$
Note, that for every $f\in E$ the modulus $\Omega(f,s)$ goes to zero when $s$ goes to zero.  Below we are using  an easy verifiable inequality
$
\Omega\left(f, as\right)\leq \left(1+a\right)\Omega(f,
s),\>\>\> a\in \mathbb{R}_{+}.
$
We obtain
$$
\|f-\mathcal{R}_{h}^{\sigma}(f)\|\leq
\int_{-\infty}^{\infty}h(t)\left\|\Delta_{t/\sigma}f\right\|dt\leq
\int_{-\infty}^{\infty}h(t)\Omega\left(f, t/\sigma\right)dt\leq 
$$
$$
\Omega\left(f,
\sigma^{-1}\right)\int_{-\infty}^{\infty}h(t)(1+|t|)dt\leq
{C}_{h}\Omega\left(f,
\sigma^{-1}\right),
$$
where the integral
 $$
C_{h}=\int_{-\infty}^{\infty}h(t)(1+|t|)dt
$$
 is finite by the choice of $h$. Theorem  is proved.
\end{proof}

\subsection{Landau-Kolmogorov-Stein-type inequality}

Let us introduce the Favard constants (see \cite{Akh}, Ch. V)
which are defined as
$$
K_{j}=\frac{4}{\pi}\sum_{r=0}^{\infty}\frac{(-1)^{r(j+1)}}{(2r+1)^{j+1}},\>\>\>
j,\>\>r\in \mathbb{N}.
$$
It is known \cite{Akh}, Ch. V, that the sequence of all Favard
constants with even indices is strictly increasing and belongs to
the interval $[1,4/ \pi)$ and the sequence of all Favard constants
with odd indices is strictly decreasing and belongs to the
interval $(\pi/4, \pi/2],$ i.e.,
\begin{equation}
K_{2j}\in [1,4/ \pi), \; K_{2j+1}\in (\pi/4, \pi/2].\label{Fprop}
\end{equation}
We will need the following generalization of the classical
Landau-Kolmogorov-Stein inequality. 
\begin{lemma} 
For every  $f\in \mathcal{D}^{\infty} $ 
the following inequalities  hold true for $0\leq k \leq n$
\begin{equation}
\left\|D^{k}f\right\|^n \leq C_{k,n}\|D^{n}f\|^{k}\|f\|^{n-k},\>\>\>
0\leq k \leq n,\label{KS}
\end{equation}
where $C_{k,n}= (K_{n-k})^n/(K_{n})^{n-k}.$ \label{LKS}
\end{lemma}
\begin{proof}
Indeed, for any $h^{*}\in E^{*}$ the classical Landau-Kolmogorov-Stein inequality 
applied to the entire function $\left<e^{tD}f,h^{*}\right>$ gives
$$
\left\|\left(\frac{d}{dt}\right)^{k}\left<e^{tD}f,h^{*}\right>\right\|^n_{
C(\mathbb{R}^{1})}\leq
C_{k,n}\left\|\left(\frac{d}{dt}\right)^{n}\left<e^{tD}f,h^{*}\right>\right\|_{
C(\mathbb{R}^{1})}^{k}\times
$$
$$
\left\|\left<e^{tD}f,h^{*}\right>\right\|_{C(\mathbb{R}^{1})}^{n-k},\quad
0<k< n,
$$
or
$$
\left\|\left<e^{tD}D^{k}f,h^{*}\right>\right\|_{ C(\mathbb{R}^{1})}^n\leq
C_{k,n}\left\|\left<e^{tD}D^{n}f,h^{*}\right>\right\|_{
C(\mathbb{R}^{1})}^{k}\left\|\left<e^{tD}f,h^{*}\right>\right\|_{C(\mathbb{R}^{1})}^{n-k}.
$$
Applying the Schwartz inequality to the right-hand side, we obtain
\begin{align*}
\left\|\left<e^{tD}D^{k}f,h^{*}\right>\right\|_{ C(\mathbb{R}^{1})}^n & \leq
C_{k,n}\|h^{*}\|^{k}\|D^{n}f\|^{k}\|h^{*}\|^{n-k}\|f\|^{n-k}\\
&\leq C_{k,n}\|h^{*}\|^n \|D^{n}f\|^{k}\|f\|^{n-k},
\end{align*}
which, when $t=0,$ yields
$$
\left|\left<D^{k}f,h^{*}\right>\right|^n\leq C_{k,n}\|h^{*}\|^n
\|D^{n}f\|^{k}\|f\|^{n-k}.
$$
By choosing $h$ such that $\left|\left<D^{k}f,h^{*}\right>\right|=\|D^{k}f\|$ and
$\|h^{*}\|=1$  we obtain (\ref{KS}).
\end{proof}

\subsection{Another characterization of the abstract Bernstein spaces}

\begin{definition}
For a given $f\in E$ the notation $\sigma_{f}$ will be used for the smallest finite real number (if any) for which 
$$\|D^{k}f\|\leq \sigma_{f}^{k}\|f\|,\>\>\>k\in\mathbb{N}.
$$
 If there is no such finite number we set $\sigma_{f}=\infty$. 
\end{definition}

Now we  are going to prove another characterization of Bernstein spaces. In the case of Hilbert spaces corresponding result was appeard in \cite{Pes08a}, \cite{PZ}.

\begin{theorem}\label{new}
Let $f\in E$ belongs to a space  $\mathbf{B}_{\sigma}(D),$ for some
$0<\sigma<\infty.$ Then the following limit exists
 \begin{equation}
 d_f=\lim_{k\rightarrow \infty} \|D^k
f\|^{1/k} \label{limit}
\end{equation}
and $d_f=\sigma_f.$ 

Conversely, if
$f\in \mathcal{D}^{\infty}$ and $d_f=\lim_{k\rightarrow \infty}
\|D^k f\|^{1/k}$ exists and is finite, then $f\in\mathbf{B}_{d_f}(D)$
and $d_f=\sigma_f .$

\end{theorem}

\begin{proof}

From Lemma \ref{LKS} we have $$ \left\|D^{k}f\right\|^n \leq
C_{k,n}\|D^{n}f\|^{k}\|f\|^{n-k}, \quad 0\leq k \leq n.
$$
Without loss of generality, let us assume that $\|f\|=1.$ Thus,
$$ \left\|D^{k}f\right\|^{1/k} \leq
(\pi/2)^{1/kn}\|D^{n}f\|^{1/n}, \quad 0\leq k \leq n.
$$
Let $k$ be arbitrary but fixed. It follows that
$$\left\|D^{k}f\right\|^{1/k} \leq
(\pi/2)^{1/kn}\|D^{n}f\|^{1/n}, \mbox{ for all } n\geq k,$$ which
implies that
$$\left\|D^{k}f\right\|^{1/k}\leq \underline{\lim}_{n\rightarrow
\infty}\|D^{n}f\|^{1/n}.$$ But \rm since this inequality is true for
all $k>0,$ we obtain that
$$\overline{\lim}_{k\rightarrow\infty}\|D^{k}f\|^{1/k}\leq
\underline{\lim}_{n\rightarrow \infty}\|D^{n}f\|^{1/n},$$ which
proves that $d_f=\lim_{k\rightarrow}\|D^{k}f\|^{1/k}$ exists.
Because  $f\in \mathbf{B}_{\sigma}(D) $ the constant $\sigma_{f}$ is finite and we have 
$$
\|D^{k}f\|^{1/k}\leq \sigma_f  \|f\|^{1/k},
$$
 and by taking the
limit as $k\rightarrow\infty$ we obtain $d_f\leq \sigma_f.$ To show
that $d_f= \sigma_f,$ let us assume that $d_f< \sigma_f.$
There exist $M>0$ and $\sigma $ such that $0<d_f<\sigma
< \sigma_f$ and $$\|D^k f\|\leq M \sigma^k, \quad \mbox{for all }
k>0 .$$ Thus, by Lemma  \ref{basic} we have $f\in \mathbf{B}_\sigma(D) ,$
which is a contradiction to the definition of $\sigma_f.$

Conversely,  suppose that $d_f=\lim_{k\rightarrow \infty} \|D^k
f\|^{1/k}$ exists and is finite. Therefore, there exist $M>0$ and
$\sigma
>0$ such that $d_f<\sigma $ and
$$\|D^k f\|\leq M \sigma^k, \quad \mbox{for all } k>0 ,$$
which, in view of Lemma \ref{basic}, implies that $f\in
\mathbf{B}_{\sigma}(D) .$ Now by repeating the argument in the first part of
the proof we obtain $d_f=\sigma_f ,$ where $\sigma_f=\inf\left\{
\sigma : f\in \mathbf{B}_{\sigma}(D)\right\}.$

The proof is complete. 

\end{proof}

\subsection{Bernstein vectors in Hilbert spaces}

Let $D$ be a self-adjoint operator in a Hilbert space $H$. 
 According to the spectral theory
\cite{BS}, there exist a direct integral of Hilbert spaces $A=\int
A(\lambda )dm (\lambda )$ and a unitary operator $\mathcal{F}_{D}$
from $H$ onto $A$, which transforms the domain $\mathcal{D}_{k}$
of the operator $ D^{k}$ onto $A_{k}=\{a\in A|\lambda^{k}a\in A
\}$ with norm

$$\|a(\lambda )\|_{A_{k}}= \left (\int^{\infty}_{-\infty} \lambda ^{2k}
\|a(\lambda )\|^{2}_{A(\lambda )} dm(\lambda ) \right )^{1/2} $$
and  $\mathcal{F}_{D}(Df)=\lambda (\mathcal{F}_{D}f), f\in
\mathcal{D}_{1}. $
\begin{definition}
The unitary operator $\mathcal{F}_{D}$ will be called the Spectral
Fourier transform and $a=\mathcal{F}_{D}f $ will be called the
Spectral Fourier transform of $f\in H$. \label{Def2}
\end{definition}
\begin{definition}
We will say that a vector  $f$ in  $H$  belongs to the space
$PW_{\sigma}(D),\>\sigma>0,$  if its Spectral Fourier transform
$\mathcal{F}_{D}f=a$ has support in $[-\sigma , \sigma ] $.
\label{DefSFT}
\end{definition}
The operator $iD$ generates in $H$ a strongly continuous group of unitary operators $e^{itD},\>\>t\in \mathbb{R}$.
The
next theorem, whose proof can be found in \cite{Pes88},
\cite{Pes00}, shows that the space $PW_\sigma (D)$ coincides with the space ${\bf B}_{\sigma}(D)$.

\begin{theorem}
The following conditions are equivalent:

1)$f\in PW_{\sigma}(D)$;

2) $f$ belongs to the set
$$
\mathcal{D}_{\infty}=\bigcap_{k=1}^{\infty}\mathcal{D}_{k},
$$
and for all $k\in \mathbb{N},$ the following Bernstein inequality
holds true
\begin{equation}
\|D^{k}f\|\leq \sigma^{k}\|f\|\label{BI}.
\end{equation}

\end{theorem}

\section{Sampling formulas for orbits  of Bernstein vectors}\label
{sampling in Banach spaces}

We assume that $D$ generates one-parameter strongly continuous  group of isometries  $e^{tD}, \>\>t\in \mathbb{R},$ in a Banach space $E$. In this section we prove explicit formulas for a  trajectory $e^{tD}f$ with $f\in \mathbf{B}_{\sigma}(D)$ in terms of a countable number of equally spaced samples.

By using Theorem \ref{smaller sampling} and Corollary \ref{key-col} we obtain 
our first sampling theorem. 
\begin{theorem}\label{FST}If $f\in {\bf B}_{\sigma}(D)$ then for every $g^{*}\in E^{*}\>$, every $0<\gamma< 1,$ and every $z\in \mathbb{C}$ one has
\begin{enumerate}
\item The real variable function $\langle e^{tD}f,\>g^{*}\rangle$ belongs to the regular space $B_{\sigma}^{\infty}(\mathbb{R})$, which means it has extension $\langle e^{tD}f,\>g^{*}\rangle$ to the complex plane $\mathbb{C}$ as an entire function of exponential type $\sigma$.

\item The following sampling formula holds
$$
\langle e^{zD}f,\>g^{*}\rangle=\sum_{k\in \mathbb{Z}}\left\langle e^{(\gamma k)D}f,\>g^{*}\right\rangle sinc\left(\gamma^{-1}z -k\right),\>\>z\in \mathbb{C},
$$
where the series converges uniformly on compact subsets of $\mathbb{C}$.
\end{enumerate}
\end{theorem}

\begin{theorem}\label{8.1}
If $f\in \mathbf{B}_{\sigma}(D)$ then the following sampling formulas hold for $t\in \mathbb{R}$
\begin{equation}\label{s1}
e^{tD}f=f+tDf \rm {sinc}\left(\frac{\sigma t}{\pi}\right)+t\sum_{k\neq 0}\frac{e^{\frac{k\pi}{\sigma}D}f-f}{\frac{k\pi}{\sigma}} \rm sinc\left(\frac{\sigma t}{\pi}-k\right),
\end{equation}
 
\begin{equation}\label{l0}
f=e^{tD}f-t \left(e^{t D}Df\right)\rm sinc\left(\frac{\sigma t}{\pi}\right)-t\sum_{k\neq 0}\frac{e^{\left(   \frac{k\pi}{\sigma}+t  \right)D}f-e^{tD}f}{\frac{k\pi}{\sigma}}    \rm sinc\left(\frac{\sigma t}{\pi}+k\right),
\end{equation}
where both series converge in the norm of $E$ uniformly with respect to $t\in \mathbb{R}$.
\end{theorem}

\begin{remark}

It is worth to note that if $\>\>\>t\neq  0,\>\>$ then right-hand side of (\ref{l0}) does not contain vector $f$ and we obtain a "linear combination"  of $f$ in terms of vectors  $e^{tD}Df$ and $e^{\left(   \frac{k\pi}{\sigma}+t  \right)D}f,\>\>\>k\in \mathbb{Z}$. 
Also, since integers are zeros of $ \rm sinc \>\>(t)$ both formulas are trivial for $\>\>t=\frac{\pi}{\sigma}m, \>\>\>m\in \mathbb{Z}$ (interpolation).
\end{remark}
       \begin{proof}
       
         If $f\in \mathbf{B}_{\sigma}(D)$ then for any $g^{*}\in E^{*}$ the function $F(t)=\left<e^{tD}f,\>g^{*}\right>$ belongs to $B_{\sigma}^{\infty}(\mathbb{R})$.

We consider $F_{1}\in B_{\sigma}^{2}( \mathbb{R}),$ which is defined as follows.
If $t\neq 0$ then 
\begin{equation}\label{F}
F_{1}(t)=\frac{F(t)-F(0)}{t}=\left<\frac{e^{tD}f-f}{t},\>g^{*}\right>,
\end{equation}
and if $t=0$ then
$
F_{1}(t)=\frac{d}{dt}F(t)|_{t=0}=\left<Df,\>g^{*}\right>.
$
We have
$$
F_{1}(t)=\sum_{k}F_{1}\left(\frac{k\pi}{\sigma}\right)\> \rm sinc\left(\frac{\sigma t}{\pi}-k\right),
$$
which means that   for any $g^{*}\in E^{*}$ 
$$
\left<    \frac{e^{tD}f-f}{t},\>g^{*}    \right>=\sum_{k}\left<\frac{e^{\frac{k\pi}{\sigma}D}f-f}{\frac{k\pi}{\sigma}},\>g^{*}\right>\> \rm sinc\left(\frac{\sigma t}{\pi}-k\right).
$$
The formula (\ref{sinc-1}) shows that $| sinc^{(m)}(x)|$ is of order $O(x^{-1})$ for every $m\in \mathbb{N}$ and this fact implies convergence in $E$ of the series 
$$
\sum_{k}\frac{e^{\frac{k\pi}{\sigma}D}f-f}{\frac{k\pi}{\sigma}} \rm sinc\left(\frac{\sigma t}{\pi}-k\right).
$$
It leads to the following equality for any $g^{*}\in E^{*}$ 
$$
\left<    \frac{e^{tD}f-f}{t},\>g^{*}    \right>=\left<\sum_{k}\frac{e^{\frac{k\pi}{\sigma}D}f-f}{\frac{k\pi}{\sigma}} \rm sinc\left(\frac{\sigma t}{\pi}-k\right),\>g^{*}\right>,\>\>\>t\neq 0,
$$
and if $t= 0$ it gives the identity
$
\left<Df,\>g^{*}\right>=\left<Df,\>g^{*}\right>\sum_{k} \rm sinc \>k.
$
Thus one obtains
\begin{equation}
   \frac{e^{tD}f-f}{t}=\sum_{k}\frac{e^{\frac{k\pi}{\sigma}D}f-f}{\frac{k\pi}{\sigma}} \rm sinc\left(\frac{\sigma t}{\pi}-k\right),\>\>\>t\neq 0,
\end{equation}
or for every $t\in \mathbb{R}$ 
$$
e^{tD}f=f+tDf \rm{sinc}\left(\frac{\sigma t}{\pi}\right)+t\sum_{k\neq 0}\frac{e^{\frac{k\pi}{\sigma}D}f-f}{\frac{k\pi}{\sigma}} \rm sinc\left(\frac{\sigma t}{\pi}-k\right).
$$

We proved the formula (\ref{s1}). 
If in (\ref{s1}) we replace $f$ by $\left(e^{\tau D}f\right)$ for a $\tau\in \mathbb{R}$ we will have

\begin{equation}\label{s100}
e^{tD}\left(e^{\tau D}f\right)=
$$
$$
\left(e^{\tau D}f\right)+t\sum_{k\neq 0}\frac{e^{\frac{k\pi}{\sigma}D}\left(e^{\tau D}f\right)-\left(e^{\tau D}f\right)}{\frac{k\pi}{\sigma}} \rm sinc\left(\frac{\sigma t}{\pi}-k\right)+tD\left(e^{\tau D}f\right)\rm sinc\left(\frac{\sigma t}{\pi}\right).
\end{equation}
For $t=-\tau $ we obtain the next formula which holds for any $\tau\in \mathbb{R},\>\>f\in \mathbf{B}_{\sigma}(D),$ 
\begin{equation}\label{l1}
f=e^{\tau D}f-\tau \sum_{k\neq 0}\frac{e^{\left(   \frac{k\pi}{\sigma}+\tau  \right)D}f-e^{\tau D}f}{\frac{k\pi}{\sigma}}    \rm sinc\left(\frac{\sigma \tau}{\pi}+k\right)-\tau D\left(e^{\tau D}f\right)\rm sinc\left(\frac{\sigma \tau}{\pi}\right),
\end{equation}
which is the formula (\ref{l0}). 
This completes the proof of the theorem.
\end{proof}

The next theorem is a generalization of the Valiron-Tschakaloff
sampling/interpolation Theorem \ref{classicalVT}.

\begin{theorem}\label{groupVT}
For $f\in \mathbf{B}_{\sigma}(D),\>\>\>\sigma>0,$ we have for all $z \in \mathbb{C}$
\begin{equation}\label{VT}
e^{zD}f=z \>\rm sinc\left(\frac{\sigma  z}{\pi}\right)Df+\rm sinc\left(\frac{\sigma  z}{\pi}\right)f+\sum_{k\neq 0}\frac{\sigma z}{k\pi}\rm sinc\left(\frac{\sigma  z}{\pi}-k\right)e^{\frac{k\pi}{\sigma}D}f,
\end{equation}
where the series converges in the norm of $E$.
\end{theorem}

\begin{proof} Since for  any $g^{*}\in E^{*}$ the function $F(t)=\left<e^{tD}f,\>g^{*}\right>$ is in 
$ B_{\sigma}^{\infty}(\mathbb{R}),\>\>\>\sigma>0,$ then for all $z \in \mathbb{C}$ the  
Valiron-Tschakaloff sampling/interpolation Theorem formula \ref{classicalVT} gives

\begin{equation}
\left<e^{zD}f,\>g^{*}\right>=z \>\rm sinc\left(\frac{\sigma  z}{\pi}\right)\left<Df,\>g^{*}\right>+
$$
$$
\rm sinc\left(\frac{\sigma  z}{\pi}\right)\left<f,\>g^{*}\right>+\sum_{k\neq 0}\frac{\sigma z}{k\pi}\rm sinc\left(\frac{\sigma  z}{\pi}-k\right)\left<e^{\frac{k\pi}{\sigma}D}f,\>g^{*}\right>.
\end{equation}
Because the series 
$$
\sum_{k\neq 0}\frac{\sigma z}{k\pi}\rm sinc\left(\frac{\sigma  z}{\pi}-k\right)e^{\frac{k\pi}{\sigma}D}f,
$$
converges in $E$ for every fixed $z$ and the equality holds for every $g^{*}\in E^{*}$ we obtain the formula (\ref{VT}).

\end{proof}

\section{Boas type formulas for orbits of Bernstein vectors}\label{RB in Banach}

Motivated by formulas (\ref{Butzer1}) and (\ref{Butzer2})  we introduce the following  bounded operators

\begin{equation}\label{b1}
\mathcal{B}_{D}^{(2m-1)}(\sigma)f=\left(\frac{\sigma}{\pi}\right)^{2m-1}\sum_{k\in \mathbb{Z}}(-1)^{k+1}A_{m,k}e^{\frac{\pi}{\sigma}(k-1/2)D}f,\>\> 
\end{equation}
where $f\in E, \>\>\sigma>0,\>\>\>m\in \mathbb{N},$ along with operators 
\begin{equation}\label{b2}
\mathcal{B}_{D}^{(2m)}(\sigma)f=\left(\frac{\sigma}{\pi}\right)^{2m}\sum_{k\in \mathbb{Z}}(-1)^{k+1}B_{m,k}e^{\frac{\pi k}{\sigma}D}f, \>\>f\in E,\> \sigma>0,\>m\in \mathbb{N},
\end{equation}
  Both series converge in $E$ due to the identities (\ref{relations}) 
and   because $\|e^{tD}f\|=\|f\|$ one obtains
 
\begin{equation}\label{norms}
 \|\mathcal{B}_{D}^{(2m-1)}(\sigma)f\|\leq \sigma^{2m-1}\|f\|,\>\>\>\>\>\|\mathcal{B}_{D}^{(2m)}(\sigma)f\|\leq \sigma^{2m}\|f\|,\>\>\>f\in E. 
 \end{equation}
 
 \begin{theorem}
 The following Boas-type interpolation formulas hold true for $r\in \mathbb{N}$

\begin{equation}\label{B1 for groups}
D^{r}f=\mathcal{B}_{D}^{(r)}(\sigma)f,\>\>\>\>\>f\in \mathbf{B}_{\sigma}(D),
\end{equation}
where convergence of the series is  in the norm of $E$ uniformly with respect to $t\in \mathbb{R}$.
 \end{theorem}

\begin{proof}
For any  $\psi^{*}\in E^{*}$
 the  function $
F(t)=\left<e^{tD}f, \psi^{*}\right>
$
is of exponential type $\sigma$ and bounded on $\mathbb{R}.$ Thus we have
$$
F^{(2m-1)}(t)=\left(\frac{\sigma}{\pi}\right)^{2m-1}\sum_{k\in \mathbb{Z}}(-1)^{k+1}A_{m,k}F\left(t+\frac{\pi}{\sigma}(k-1/2)    \right),\>\>\>m\in \mathbb{N},
$$
$$
F^{(2m)}(t)=\left(\frac{\sigma}{\pi}\right)^{2m}\sum_{k\in \mathbb{Z}}(-1)^{k+1}B_{m,k}F\left(t+\frac{\pi k}{\sigma}    \right),\>\>\>m\in \mathbb{N}.
$$
Together with 
$$
\left(\frac{d}{dt}\right)^{k}F(t)=\left<D^{k}e^{tD}f,\psi^{*}\right>,
$$
  it shows
$$
\left<e^{tD}D^{2m-1}f, \psi^{*}\right>=\left(\frac{\sigma}{\pi}\right)^{2m-1}\sum_{k\in \mathbb{Z}}(-1)^{k+1}A_{m,k}\left<e^{\left(t+\frac{\pi}{\sigma}(k-1/2)    \right)D}f,\>\>\psi^{*}\right>,\>\>\>m\in \mathbb{N},
$$
and also
$$
\left<e^{tD}D^{2m}f, \psi^{*}\right>=\left(\frac{\sigma}{\pi}\right)^{2m}\sum_{k\in \mathbb{Z}}(-1)^{k+1}B_{m,k}\left<e^{\left(t+\frac{\pi k}{\sigma}    \right)D}f,\>\> \psi^{*}\right>,\>\>\>m\in \mathbb{N}.
$$
Since both series  converge in $E$, and the last two equalities hold for any $\psi^{*}\in E$ we obtain the next two formulas

\begin{equation}\label{sam1}
e^{tD}D^{2m-1}f=\left(\frac{\sigma}{\pi}\right)^{2m-1}\sum_{k\in \mathbb{Z}}(-1)^{k+1}A_{m,k}e^{\left(t+\frac{\pi}{\sigma}(k-1/2)   \right)D}f,\>\>\>m\in \mathbb{N},
\end{equation}
\begin{equation}\label{sam2}
e^{tD}D^{2m}f=\left(\frac{\sigma}{\pi}\right)^{2m}\sum_{k\in \mathbb{Z}}(-1)^{k+1}B_{m,k}e^{\left(t+\frac{\pi k}{\sigma}    \right)D}f, \>\>\>m\in \mathbb{N}.
\end{equation}
In turn when $t=0$ these formulas become formulas (\ref{B1}).  
Theorem is proven.
\end{proof}

As a summary of  sections \ref{Bernstein-vectors}-\ref{RB in Banach} we can formulate the following statement.

 \begin{theorem}\label{main}
 If $D$ generates a one-parameter strongly continuous  group of isometries   $e^{tD}$ in a Banach space $E$ then  the following conditions are equivalent:

\begin{enumerate}

\item  $f$ belongs to $\mathbf{B}_{\sigma}(D),\>\>0<\sigma<\infty$.

\item   The abstract-valued function $e^{tD}f$  is entire abstract-valued function of exponential type $\sigma$ which is bounded on the real line.

\item For every functional $g^{*}\in E^{*}$ the function $\left<e^{tD}f,\>g^{*}\right>$  is entire function of exponential type $\sigma$ which is bounded on the real line, i.e. it belongs to the regular $B_{\sigma}^{\infty}(\mathbb{R})$.
\item The following Boas-type interpolation formulas hold true for $r\in \mathbb{N}$

\begin{equation}\label{B1}
D^{r}f=\mathcal{B}_{D}^{(r)}(\sigma)f.
\end{equation}

\item The following limit exists
 \begin{equation}
 d_f=\lim_{k\rightarrow \infty} \|D^k
f\|^{1/k} \label{limit}<\infty,
\end{equation}
and $d_f$ coincides with the lower boundary of all $\sigma>0$  such that $f\in {\bf B}_{\sigma}(D)$.
\end{enumerate}
\end{theorem}

\section{Some harmonic analysis associated with the Discrete Hilbert transform}\label{Analysis with HT}

We will be interested in the operator $H=\pi \widetilde{H}$ where $\widetilde{H}$ is 
 the discrete Hilbert transform operator
 $$
\widetilde{H}: l^{2}\mapsto l^{2},\>\>\>\widetilde{H}{\bf a}={\bf b},\>\>\> {\bf a}=\{a_{j}\}\in l^{2},\>\>\>{\bf b}=\{b_{m}\}\in l^{2},
$$
which is defined by the formula
\begin{equation}
b_{m}=\frac{1}{\pi}\sum_{n\neq m,\>n\in \mathbb{Z}}\frac{a_{n}}{m-n},\>\>\>m\in \mathbb{Z}.
\end{equation}

As it was proved  by Schur \cite{Schur} the operator norm of $\widetilde{H}: l^{2}\mapsto l^{2}$ is one and therefore the operator norm of $H$ is $\pi$. In what follows the notation $\|\cdot\|$ will always mean $\|\cdot\|_{l^{2}}$.
It was shown in \cite{Graf}  that although the norm of the operator $H$ is $\pi$, only a strong inequality $\|H{\bf a}\|<\pi\|{\bf a}\|$ can hold
for every non-trivial ${\bf a}\in l^{2}$. 
Since $H$ is a bounded operator the following exponential series converges in $l^{2}$ for every ${\bf a}\in l^{2}$ and every $t\in \mathbb{R}$
\begin{equation}
e^{tH}{\bf a}=\sum_{k=0}^{\infty}\frac{H^{k}{\bf a}}{k!}t^{k}.
\end{equation}
Thus $H$ is a generator of a one-parameter group of operators  $e^{tH}, \>t\in \mathbb{R}$.
It is clear that for  a general bounded operator $A$ the exponent can be extended to the entire complex plane $\mathbb{C}$ and one has the estimate 
\begin{equation}\label{bound}
\|e^{zA}\|\leq \sum_{k=0}^{\infty}\frac{\|A\|^{k}|z|^{k}}{k!}= e^{\|A\||z|},\>\>\>z\in \mathbb{C}.
\end{equation}
In the paper \cite{DC} about the group $e^{tH}$ the following results were obtained:
\begin{enumerate}
\item the explicit formulas for the operators $e^{tH}$ were given,

\item it was shown that every operator $e^{tH}$ is an isometry in $l^{2}$.

\end{enumerate}
The explicit formulas are given in the next statement.
\begin{theorem}The operator $H$ generates in $l^{2}$ a one-parameter group of isometries $ e^{tH}{\bf a}={\bf b},\>\>\>{\bf a}=(a_{n})\in l^{2}, \>\>\>{\bf b}=(b_{m})\in l^{2}, $ which is given by the formulas
$$
b_{m}=\frac{\sin(\pi t)}{\pi}\sum_{n\in \mathbb{Z}}\frac{a_{n}}{m-n+t},\>\>\>
$$
if $t\in \mathbb{R}\setminus \mathbb{Z}$, and 
$$
b_{m}=(-1)^{t}a_{m+t}, \>\>
$$
if $t\in \mathbb{Z}$.

\end{theorem}

We note that since $H$ is a bounded operator whose norm is $\pi$ one has for all ${\bf a}\in l^{2}$  the following Bernstein-type  inequality 
\begin{equation}\label{bern-1}
\|H^{k}{\bf a}\|\leq \pi^{k}\|{\bf a }\|.
\end{equation}
In terminology of chapter 3 it means that 
\textit{the entire space $l^{2}$ is the Bernstein space ${\bf B}_{\pi}(H)$.}

\bigskip

The next Theorem is a consequence of Theorem \ref{main}.

\begin{theorem}\label{Hilbert-Bern-space-functional}

 The following holds true.
\begin{enumerate}

\item For every ${\bf a}\in l^{2}$
$$
e^{tH}{\bf a}: \mathbb{R}\mapsto l^{2}
$$
is an abstract valued function which has extension $ e^{z H}{\bf a},\>\>z\in \mathbb{C},$ to the complex plane as an entire function of the exponential type $\pi$ which is bounded on the real line.
\item 

For every ${\bf a},
{\bf a}^{*}\in l^{2}$  the function 
$$
\Phi(t)=\langle e^{t H}{\bf a}, {\bf a}^{*}\rangle,\>\>\>\>t\in \mathbb{R},
$$
 belongs to the Bernstein class $B_{\pi}^{\infty}(\mathbb{R}).$ 

\end{enumerate}
 
 \end{theorem}

\section{Sampling theorems  for orbits of the one-parameter group generated by the Discrete Hilbert transform}\label{H-sampling}

By using Theorem \ref{smaller sampling} and Theorem \ref{Hilbert-Bern-space-functional} we obtain 
our first sampling theorem. 
\begin{theorem}\label{FST}For  every ${\bf a},\>{\bf a}^{*}\in l^{2}$, every $0<\gamma< 1,$ and every $z\in \mathbb{C}$ one has

\begin{equation}\label{1-2}
\left\langle e^{zH}{\bf a}, {\bf a}^{*}\right\rangle=\sum_{k\in \mathbb{Z}} \left\langle e^{(\gamma k) H }{\bf a}, {\bf a}^{*}\right\rangle\>
sinc\left(\gamma^{-1}z -k\right),
\end{equation}
where the series converges uniformly on compact subsets of $\mathbb{C}$.
\end{theorem}

Explicitly, the formula (\ref{1-2}) means that 
if $z$ in (\ref{1-2}) is a real $z=t$ which is not an integer then 
  (\ref{1-2}) takes the form
\begin{equation}
\frac{\sin \pi t}{\pi}\sum_{m\in \mathbb{Z}} \sum_{n\in \mathbb{Z}\>\>n\neq m}   \frac{a_{n}b_{m}} {m-n+t}=S_{1}+S_{2}, 
\end{equation}
where
$$
S_{1}=\sum_{k\in \mathbb{Z},\> \gamma k\in \mathbb{R}\setminus\mathbb{Z}}\frac{\sin \pi \gamma k}{\pi}\left(    \sum_{m\in \mathbb{Z}}\sum_{n\in \mathbb{Z}\>\>n\neq m}  \frac{a_{n}b_{m}}{m-n+\gamma k      } \right)\>sinc(\gamma^{-1}t-k),
$$
and 
$$
S_{2}=\sum_{k\in \mathbb{Z};\>\gamma k\in\mathbb{Z}} \left( \sum_{m\in \mathbb{Z}}(-1)^{\gamma k}a_{m+\gamma k}b_{m}\right)\>
sinc(\gamma^{-1}t-k).
$$

Next, if $z=t$ in (\ref{1-2}) is an integer   then the formula  (\ref{1-2}) is given by 
$$
\sum_{m\in \mathbb{Z}}(-1)^{t}a_{m+t}b_{m}=S_{1}+S_{2}.
$$
We note, that if $t=\gamma N,\>\>N\in \mathbb{Z}$, then (\ref{1-2}) is evident since its both sides are (obviously) identical:
$$
\langle e^{tH}{\bf a}, {\bf a}^{*}\rangle=\langle e^{tH}{\bf a}, {\bf a}^{*}\rangle.
$$
The next statement is a consequence of Theorem \ref{8.1}.
\begin{theorem}\label{SST}For every ${\bf a}\in l^{2}$
\begin{equation}\label{2-2}
e^{tH}{\bf a}=
 {\bf a}+t\>sinc \>\left(t\right) H{\bf  a}\>+
t\sum_{k\in \mathbb{Z}\setminus \{0\}}  \frac{e^{kH}{\bf a}-{\bf a}}{k}
\>
sinc\left(t -k\right),
\end{equation}
where the series converges in the norm of $l^{2}$ uniformly with respect to $t\in \mathbb{R}$.

\end{theorem}

Next, we reformulate (\ref{2-2}) in its "native" terms.

\begin{theorem}
If ${\bf a }=(a_{n})\in l^{2}$  and $t$ is not integer then the left-hand side of (\ref{2-2}) is a sequence $e^{tH}{\bf a}={\bf b}=(b_{m})\in l^{2}$ 
with the entries 
\begin{equation}\label{Sec-1}
b_{m}=\frac{\sin \pi t}{\pi}\sum_{n\in \mathbb{Z}}   \frac{a_{n}} {m-n+t},
\end{equation}
and the right-hand side represents a sequence ${\bf c}=(c_{m})\in l^{2}$  with the entries
\begin{equation}\label{sec-2}
c_{m}=a_{m}+
t\>sinc(t)\sum_{n\in \mathbb{Z}, n\neq m}\frac{a_{n}}{m-n}+
$$
$$
t\sum_{k\in \mathbb{Z}\setminus \{0\}}  
\frac{  (-1)^{k}a_{m+k}-a_{m}}{k}\>
sinc\left(t -k\right).
\end{equation}
If $t$ in  (\ref{2-2}) is an integer $t=N$  then $b_{m}=(-1)^{N}a_{m+N}$ and 
$$
c_{m}=a_{m}+
N\sum_{n\in \mathbb{Z}, n\neq m}\frac{a_{n}}{m-n}\>sinc(N)+
$$
$$
N\sum_{k\neq 0}  
\frac{  (-1)^{k}a_{m+k}-a_{m}}{k}\>
sinc\left(N -k\right)=(-1)^{N}a_{m+N}.
$$
Thus in the case when $t=N$ is an integer we obtain just a tautology
$$
b_{m}=(-1)^{N}a_{m+N}=c_{m}.
$$
\end{theorem}

The Valiron-Tschakaloff Theorem for groups \ref{groupVT} implies the following statement.

\begin{theorem}\label{V-T}
For every ${\bf a}\in l^{2}$
one  has 
\begin{equation}\label{VT-2}
e^{tH}{\bf a}=sinc\left( t\right){\bf a}+t \>sinc\left( t\right)H{\bf a}+  \sum_{k\in \mathbb{Z}\setminus \{0\}}\frac{ t}{k}sinc\left( t-k\right)e^{kH}{\bf a},
\end{equation}
where the series converges in the norm of $l^{2}$ uniformly with respect to $t\in \mathbb{R}$.
\end{theorem}

The following is a reformulation of  (\ref{VT-2}) in the specific language of $l^{2}$.

\begin{theorem}
If ${\bf a }=(a_{n})\in l^{2}$  and $t$ is not integer then the left-hand side of (\ref{VT-2}) is a sequence $e^{tH}{\bf a}={\bf b}=(b_{m})\in l^{2}$ 
with entries 

\begin{equation}\label{Sec-1}
b_{m}=\frac{\sin \pi t}{\pi}\sum_{n\in \mathbb{Z}}   \frac{a_{n}} {m-n+t},
\end{equation}
and the right-hand side of (\ref{VT-2}) represents a sequence ${\bf c}=(c_{m})\in l^{2}$  with entries
\begin{equation}\label{sec-2}
c_{m}=a_{m}\>sinc(t)+
$$
$$
t\>sinc(t) 
\sum_{n\in \mathbb{Z}, n\neq m}\frac{a_{n}}{m-n}+
t\sum_{k\in \mathbb{Z}\setminus \{0\}} (-1)^{k}\frac{ sinc\left(t -k\right) }{k  } a_{m+k}.
\end{equation}
\end{theorem}
When $t$ in  (\ref{VT-2}) is an integer $t=N$  then (\ref{VT-2}) is the tautology $b_{m}=(-1)^{N}a_{m+N}=c_{m}$.

\section{Boas interpolation formulas for the orbits of the one-parameter group generated by   the Discrete Hilbert transform}\label{H-Boas formulas}

Let's introduce bounded operators
\begin{equation}\label{b1}
\mathcal{R}^{(2s-1)}_{H}{\bf a}=
\sum_{k\in \mathbb{Z}}(-1)^{k+1}A_{s,k}
e^{(k-1/2)H}{\bf a}
,\>\> \>\>{\bf a}\in l^{2},\>\>\>s\in \mathbb{N},
\end{equation}
and
\begin{equation}\label{b2}
\mathcal{R}^{(2s)}_{H}{\bf a}=
\sum_{k\in \mathbb{Z}}(-1)^{k+1}B_{s,k}
e^{kH}
{\bf a}, \>\> \>\>{\bf a}\in l^{2}, \>\>s\in \mathbb{N}.
\end{equation}

Theorem \ref{B1 for groups} implies. the following one.

 \begin{theorem}
 For ${\bf a}\in l^{2}$  the following Boas-type interpolation formulas hold true for $r\in \mathbb{N}$
\begin{equation}\label{B1}
H^{r}{\bf a}=\mathcal{R}^{(r)}_{H}{\bf a},\>\>\>{\bf a}\in l^{2}.
\end{equation}
More explicitly, if $r=2s-1, \>s\in \mathbb{N},$ then 
\begin{equation}\label{RB1}
H^{2s-1}{\bf a}=\sum_{k\in \mathbb{Z}}(-1)^{k+1}A_{s,k}
e^{(k-1/2)H}{\bf a},
\end{equation}
and when $r=2s,\>s\in \mathbb{N},$ then 
\begin{equation}\label{RB2}
H^{2s}{\bf a}=
\sum_{k\in \mathbb{Z}}(-1)^{k+1}B_{s,k}
e^{kH}
{\bf a}.
\end{equation}
\end{theorem}

 Let us introduce the  notation
 $$
  \mathcal{R}_{H}= \mathcal{R}_{H}^{(1)}.
 $$
 One has the following "power" formula which easily follows from the fact that operators $
  \mathcal{R}_{H}$ and $H$ commute.
 \begin{col}
 For any $r\in \mathbb{N}$ and  any ${\bf a}\in l^{2}$
 \begin{equation}\label{powerB}
 H^{r}{\bf a}=\mathcal{R}_{H}^{(r)}{\bf a}=\mathcal{R}_{H}^{r}{\bf a},
 \end{equation}
 where $\mathcal{R}_{H}^{r}{\bf a}=\mathcal{R}_{H}\left(... \left(\mathcal{R}_{H}{\bf a}\right)\right).$
\end{col}

Let's express (\ref{B1}) in terms of $H$ and $e^{tH}$. Our  starting sequence is  ${\bf a}=(a_{n})$ and  then 
we use the notation $H^{k}{\bf a}=\left(a^{(k)}_{m}\right),\>k\in \mathbb{Z}$. One has

$$
H{\bf a}=\frac{1}{\pi}\sum_{(n,\>n\neq n_{1})}\frac{a_{n}}{n_{1}-n}=\left( a^{(1)}_{n_{1}} \right),
$$
$$
H^{2}{\bf a}=
\frac{1}{\pi^{2}}\sum_{(n_{1},\>n_{1}\neq n_{2})}\sum_{(n,\>n\neq n_{1})}\frac{a^{(1)}_{n_{1}}}{(n_{2}-n_{1})(n_{1}-n)}=\left(a^{(2)}_{n_{2}}\right)
$$
and so on up to a $r\in \mathbb{N}$
\begin{equation}\label{gen-term}
H^{r}{\bf a}=
\frac{1}{\pi^{r}}\sum_{(n_{r-1},\>n_{r-1}\neq n_{r})} \sum_{(n_{r-2},\>n_{r-2}\neq n_{r-1})}...
$$
$$
 ...\sum_{(n_{1},\>n_{1}\neq n_{2})} \sum_{(n,\>n\neq n_{1})}\frac{a^{(r-1)}_{n_{r-1}}}{(n_{r}-n_{r-1}) (n_{r}-n_{r-1})... (n_{2}-n_{1})(n_{1}-n)}=\left(a^{(r)}_{n_{r}}\right).
\end{equation}
\begin{theorem}
For any ${\bf a}=(a_{n})\in l^{2},\>\>$ and $r=2s-1$ we have in (\ref{RB1}) equality of two sequences where on the left-hand side  we have a sequence whose general term is  
$a^{(2s-1)}_{m}$,
and on the right-hand side we have a sequence whose general term is $c_{m,s}$ where 
$$
c_{m,s}=\sum_{k\in \mathbb{Z}}(-1)^{k+1}
\frac{\sin(\pi (k-1/2))}{\pi} A_{s,k}\sum_{n\neq m}\frac{a_{n}}{m-n+(k-1/2)}.
$$
The equality (\ref{RB1}) tells  that $\>\>a^{(2s-1)}_{m}=c_{m,s}$.

For the case $r=2s$ a sequence on the left hand-side of (\ref{RB2}) has a general term $a^{(2s)}_{m}$, and a sequence on the right has a general term $d_{m,s}$ of the form
$$
d_{m,s}=-\sum_{k\in \mathbb{Z}}B_{s,k}a_{m+k},
$$
and (\ref{RB2}) means that $\>\>a^{(2s)}_{m}=d_{m,s}$.
\end{theorem}


\begin{thebibliography}{}


\bibitem{Akh}
J. ~Akhiezer, {\em Theory of approximation}, Ungar, NY, 1956.

\bibitem{And}
N.~Andersen, {\em  Entire $L_{p}$-functions of exponential type}, Expo. Math. 32 (2014), no. 3, 199-220. 

\bibitem{Bang1} H. H. ~Bang,  { \em Functions with bounded spectrum, }
Trans. Amer. Math. Soc., 347(1995), no. 3, pp. 1067-1080.

\bibitem{Bang2}
H.H.~Bang, V.H.~Huy {\em Bernstein inequality for multivariate functions with smooth Fourier images}, Reprint of Ukraïn. Mat. Zh. 74 (2022), no. 11, 1558-1570. Ukrainian Math. J. 74 (2023), no. 11, 1780-1794.

\bibitem{Bard1}
C. ~Bardaro,  P.L.~Butzer,  I. Mantellini, G. Schmeisser,  R.L. Stens, {\em Classical and approximate exponential sampling formula: their interconnections in uniform and Mellin-Lebesgue norms},  Sampling, approximation, and signal analysis—harmonic analysis in the spirit of J. Rowland Higgins, 3–22, Appl. Numer. Harmon. Anal., Birkhäuser/Springer, Cham, [2023], ©2023.

\bibitem{Bard2}
C.~Bardaro, P.L.~Butzer, I.~Mantellini, G.~Schmeisser, {\em Valiron's interpolation formula and a derivative sampling formula in the Mellin setting acquired via polar-analytic functions},  Comput. Methods Funct. Theory 20 (2020), no. 3-4, 629-652.
\bibitem{BS}
M. ~Birman and M. Solomyak, {\em Spectral theory of self-adjoint
operators in Hilbert space}, D. Reidel Publishing Co., Dordrecht,
1987.

\bibitem{B1} 
R. ~Boas,  {\em The derivative of a trigonometric integral},  J. Lond. Math. Soc. , 164, 1937, 1-12.



\bibitem{B2} R. ~Boas, {\em Entire Functions,} Academic Press,
New York (1954).



\bibitem{BB}
P. ~Butzer, H. ~Berens, {\em Semi-Groups of operators and
approximation}, Springer, Berlin, 1967 .


\bibitem{BN}
P. Butzer,  R. Nessel,  {\em Fourier Analysis and Approximation}; Academic Press: New York, NY,
USA; Birkhäuser: Basel, Switzerland, 1971.

\bibitem{BSS-1}
P.L. ~Butzer, W. ~Splettstößer,  R.L. ~Stens,  {\em  The sampling theorem and linear prediction in signal
analysis},  Jahresber. Deutsch. Math.-Verein. 1988, 90, 1-70.

\bibitem{BSS-2} 
P.L. ~Butzer,  G. ~Schmeisser, R.L.~Stens, {\em An Introduction to Sampling Analysis}, In Nonuniform
Sampling-Thery and Practice; Marvasti, F., Ed.; Kluwer/Plenum: New York, NY, USA, 2001;
pp. 17-121.


\bibitem{BFHSS}

P.L. Butzer, P.J.S.G. Ferreira, J.R. Higgins, G. Schmeisser, 
R.L. Stens, {\em      The sampling theorem, Poissonos summation formula, general
Parseval formula, reproducing kernel formula and the Paley-Wiener theorem for bandlimited signals- their interconnections}, Applicable Analysis
Vol. 90, Nos. 3-4, March-April 2011, 431-461.


\bibitem{BSS-3}
P. L. Butzer, G. Schmeisser,  R.L. Stens, {\em Shannon's Sampling Theorem for Bandlimited Signals and Their
Hilbert Transform, Boas-Type Formulae for Higher Order
DerivativesÑThe Aliasing Error Involved by Their Extensions
from Bandlimited to Non-Bandlimited Signals}, Entropy 2012, 14, 2192-2226; doi:10.3390/e14112192.


\bibitem {DC}
L. De Carli, S. Samad, {\em One-parameter groups and discrete Hilbert transform}, Can. Math. Bull. 59, 497-507 (2016) 


\bibitem{Ganz}
M. ~Ganzburg, {\em Discretization theorems for entire functions of exponential type}, J. Math. Anal. Appl. 550 (2025), no. 1, Paper No. 129510, 33 pp.

\bibitem{Graf}
L. Grafakos, An elementary proof of the square summability of the discrete Hilbert transform, Amer. Math. Monthly 101 (1994), 456-458.

\bibitem{Gro}
K.~Gr\"ochenig, A.~Klotz, {\em Necessary density conditions for sampling and interpolation in spectral subspaces of elliptic differential operators}, Anal. PDE 17 (2024), no. 2, 587-616.

\bibitem{K} S. G. Krein,  {\em Linear differential equations in Banach space}, Translated from the Russian by J. M. Danskin. Translations of Mathematical Monographs, Vol. 29. American Mathematical Society, Providence, R.I., 1971. v+390 pp. 


\bibitem{KPes}
S.~Krein, I.~Pesenson, {\em Interpolation Spaces and Approximation
on Lie Groups}, The Voronezh State University, Voronezh, 1990, 100 pp.

\bibitem{Mon1}
A.~Monguzzi, M.M.~Peloso, M. ~Salvatori, {\em Sampling in spaces of entire functions of exponential type in $\mathbb{C}^{n+1}$}, J. Funct. Anal. 282 (2022), no. 6, Paper No. 109377, 33 pp.

\bibitem{Mon2}A.~Monguzzi, M.M.~Peloso, M. ~Salvatori,
 {\em Fractional Paley-Wiener and Bernstein spaces},  Collect. Math. 72 (2021), no. 3, 615–643.
\bibitem{Nik} S. M. Nikol'skii, {\em Approximation of Functions of
Several Variables and Imbedding Theorems,} Springer Verlag, New
York (1975).

\bibitem{Pes88}
I.Z. ~Pesenson,{\em  Best approximations in a space of the representation of a Lie group}, (Russian) Dokl. Akad. Nauk SSSR 302 (1988), no. 5, 1055–1058; translation in Soviet Math. Dokl. 38 (1989), no. 2, 384-388. 
\bibitem{Pes00}
 I. ~Pesenson, {\em  A sampling theorem on homogeneous manifolds},
  Trans. Amer. Math. Soc.,  Vol. { 352}(9),(2000), 4257-4269.

\bibitem{Pes21}
I.Z. ~Pesenson, {\em Sampling of band-limited vectors}, J. Fourier Anal. Appl. 7 (2001), no. 1, 93-100.

\bibitem{Pes08}
I. ~Pesenson, {\em Bernstein-Nikolskii inequalities  and Riesz
interpolation formula on compact homogeneous manifolds},  J. 
Approx. Theory 150, (2008), no. 2, 175--198.


\bibitem{Pes08a}
I. ~Pesenson, {\em  A discrete Helgason-Fourier transform for Sobolev and Besov functions on non-compact symmetric spaces},  Radon transforms, geometry, and wavelets, 231-247, Contemp. Math., 464, Amer. Math. Soc., Providence, RI, 2008.

\bibitem{Pes09}
I. ~Pesenson, {\em Paley-Wiener approximations and multiscale approximations in Sobolev and Besov spaces on manifolds}, J. Geom. Anal.  19  (2009),  no. 2, 390--419.

\bibitem{PZ}
I. ~Pesenson, A. ~Zayed, {\em Paley-Wiener subspace of vectors in a  Hilbert space with applications
to integral transforms}, J.  Math. Anal. Appl. 353 (2009) 566-582.



\bibitem{Pes11}
I. Pesenson, M. Pesenson, {\em Approximation of Besov vectors by Paley-Wiener vectors in Hilbert spaces},
  Approximation Theory XIII: San Antonio 2010 (Springer Proceedings in Mathematics), by Marian Neamtu and Larry Schumaker, 249-263.

\bibitem{Pes15}
I. ~Pesenson, {\em Sampling formulas for groups of operators in Banach spaces}, Sampl. Theory Signal Image Process. 14 (2015), no. 1, 1-16.

\bibitem{Pes14}
I. ~Pesenson, {\em  Boas-type formulas and sampling in Banach spaces with applications to analysis on manifolds}, New perspectives on approximation and sampling theory, 39–61, Appl. Numer. Harmon. Anal., Birkhäuser/Springer, Cham, 2014.

\bibitem{Pes22}
I.Z.~ Pesenson, {\em Jackson-type inequality in Hilbert spaces and on homogeneous manifolds},  Anal. Math. 48 (2022), no. 4, 1153-1168. 
\bibitem{Pes23a} 
I.Z.~Pesenson, {\em Bernstein spaces, sampling, and Riesz-Boas interpolation formulas in Mellin analysis},  Sampling, approximation, and signal analysis—harmonic analysis in the spirit of J. Rowland Higgins, 69-88, Appl. Numer. Harmon. Anal., Birkhäuser/Springer, Cham, [2023], ©2023.
\bibitem{Pes23}
I.Z. ~Pesenson,  {\em Sampling and interpolation for the discrete Hilbert and Kak-Hilbert transforms}, Canad. Math. Bull. 66 (2023), no. 2, 395-410.

\bibitem{R1}
M. ~Riesz,  {\em Eine trigonometrische Interpolationsformel und einige Ungleichungen f\"{u}r Polynome}, 
Jahresber. Deutsch. Math.-Verein. 1914, 23, 354-368.

\bibitem{R2}
M. ~Riesz, {\em Les fonctions conjuguees et les series de Fourier}, C.R. Acad. Sci. 178, 1924, 1464-1467.

\bibitem{Schm}G. ~Schmeisser, {\em Numerical differentiation inspired by a formula of R. P. Boas},  J. Approx. Theory
160 (2009), 202-222.


 \bibitem{Schur}
 I. Schur, {\em Bemerkungen zur Theorie der beschra\"nkten Bilinearformen mit unendlich vielen Vera\"nderlichen}, Journal f. Math. 140 (1911), 1-28.

\bibitem{Stein} E. M.   ~Stein, {\em Functions of exponential type}, 
Annals. Math., Vol.65, (1957), pp. 582-592.


\bibitem{TZ}
V. K. ~Tuan and  A. I. ~Zayed, { \em Paley-Wiener-type theorems
for a class of integral transforms,}  J. Math. Anal. Appl., 266
(2002), no. 1, 200--226.

\bibitem{AZ} A.I. ~Zayed, {\em Advances in Shannon's Sampling Theory,} CRC Press, Boca
Raton, 1993.



\end{thebibliography}
\end{document}